\documentclass[%
  a4paper,
  onecolumn,
]{mypreprint}

\usepackage[english]{babel}

\usepackage{amsmath}
\usepackage{amssymb}
\usepackage{amsthm}

\usepackage{enumitem}

\usepackage{tikz}
  \usetikzlibrary{backgrounds}
  \usetikzlibrary{calc}
  \usetikzlibrary{decorations.markings}
  \usetikzlibrary{decorations.pathmorphing}
  \usetikzlibrary{patterns}
  \usetikzlibrary{positioning}

\usepackage{pgfplots}
  \pgfplotsset{compat = 1.13,
    colormap name = viridis,
    unbounded coords = jump
  }
  \usepgfplotslibrary{external}
  \tikzset{external/system call = {%
    pdflatex \tikzexternalcheckshellescape
      -halt-on-error
      -interaction=batchmode
      -jobname "\image" "\texsource"}}
  \tikzexternaldisable

\usepackage{caption}
\usepackage{subcaption}

\usepackage{todonotes}

\usepackage{cleveref}

\theoremstyle{plain}\newtheorem{theorem}{Theorem}
\theoremstyle{plain}\newtheorem{corollary}{Corollary}
\theoremstyle{plain}\newtheorem{proposition}{Proposition}
\theoremstyle{definition}\newtheorem{definition}{Definition}
\theoremstyle{remark}\newtheorem{remark}{Remark}

\renewcommand{\rm}[1]{\ensuremath{\mathrm{#1}}}

\newcommand{\R}{\ensuremath{\mathbb{R}}}
\newcommand{\C}{\ensuremath{\mathbb{C}}}
\newcommand{\N}{\ensuremath{\mathbb{N}}}

\newcommand{\Bu}{\ensuremath{B_{\rm{u}}}}
\newcommand{\Cp}{\ensuremath{C_{\rm{p}}}}
\newcommand{\Cv}{\ensuremath{C_{\rm{v}}}}
\newcommand{\Hpp}{\ensuremath{H_{\rm{pp}}}}
\newcommand{\Hpv}{\ensuremath{H_{\rm{pv}}}}
\newcommand{\Hvp}{\ensuremath{H_{\rm{vp}}}}
\newcommand{\Hvv}{\ensuremath{H_{\rm{vv}}}}
\newcommand{\Np}{\ensuremath{N_{\rm{p}}}}
\newcommand{\Nv}{\ensuremath{N_{\rm{v}}}}
\newcommand{\Npj}[1]{\ensuremath{N_{\rm{p}, #1}}}
\newcommand{\Nvj}[1]{\ensuremath{N_{\rm{v}, #1}}}

\newcommand{\hV}{\ensuremath{\widehat{V}}}
\newcommand{\hG}{\ensuremath{\widehat{G}}}

\newcommand{\tM}{\ensuremath{\widetilde{M}}}
\newcommand{\tD}{\ensuremath{\widetilde{D}}}
\newcommand{\tV}{\ensuremath{\widetilde{V}}}

\newcommand{\cB}{\ensuremath{\mathcal{B}}}
\newcommand{\cC}{\ensuremath{\mathcal{C}}}
\newcommand{\cH}{\ensuremath{\mathcal{H}}}
\newcommand{\cK}{\ensuremath{\mathcal{K}}}
\newcommand{\cN}{\ensuremath{\mathcal{N}}}

\newcommand{\hcB}{\ensuremath{\widehat{\mathcal{B}}}}
\newcommand{\hcC}{\ensuremath{\widehat{\mathcal{C}}}}
\newcommand{\hcH}{\ensuremath{\widehat{\mathcal{H}}}}
\newcommand{\hcK}{\ensuremath{\widehat{\mathcal{K}}}}
\newcommand{\hcN}{\ensuremath{\widehat{\mathcal{N}}}}

\newcommand{\hy}{\ensuremath{\hat{y}}}

\newcommand{\Linf}{\ensuremath{\mathcal{L}_{\infty}}}

\newcommand{\symTF}[1]{\ensuremath{G_{\mathrm{sym},#1}}}
\newcommand{\symtf}[1]{\ensuremath{g_{\mathrm{sym},#1}}}
\newcommand{\genTF}[2]{\ensuremath{G_{\mathrm{gen},#1}^{\mathrm{#2}}}}

\newcommand{\hsymTF}[1]{\ensuremath{\hG_{\mathrm{sym},#1}}}
\newcommand{\hgenTF}[2]{\ensuremath{\hG_{\mathrm{gen},#1}^{\mathrm{#2}}}}

\newcommand{\trans}{\ensuremath{\mkern-1.5mu\mathsf{T}}}
\newcommand{\herm}{\ensuremath{\mathsf{H}}}
\renewcommand{\i}{\ensuremath{\mathfrak{i}}}
\DeclareMathOperator{\dt}{d}
\DeclareMathOperator{\mspan}{span}
\DeclareMathOperator{\GP}{GP}
\DeclareMathOperator{\relerr}{relerr}
\DeclareMathOperator{\mvec}{vec}

\newcommand{\matlab}{\mbox{MATLAB}}

\newcommand{\symintVequi}{\texttt{SymInt(V, equi)}}
\newcommand{\symintVWequi}{\texttt{SymInt(VW, equi)}}
\newcommand{\symintVavg}{\texttt{SymInt(V, avg)}}
\newcommand{\symintVWavg}{\texttt{SymInt(VW, avg)}}
\newcommand{\genintVequi}{\texttt{GenInt(V, equi)}}
\newcommand{\genintVWequi}{\texttt{GenInt(VW, equi)}}
\newcommand{\genintVavg}{\texttt{GenInt(V, avg)}}
\newcommand{\genintVWavg}{\texttt{GenInt(VW, avg)}}
\renewcommand{\pod}{\texttt{POD}}
\newcommand{\podavg}{\texttt{POD(avg)}}

\definecolor{matlabblue}{HTML}{0072BD}
\definecolor{matlaborange}{HTML}{D95319}
\definecolor{matlabyellow}{HTML}{EDB120}
\definecolor{matlabpurple}{HTML}{7E2F8E}
\definecolor{matlabgreen}{HTML}{77AC30}
\definecolor{matlablightblue}{HTML}{4DBEEE}
\definecolor{matlabred}{HTML}{A2142F}

\tikzstyle{fom} = [
  black,
  solid,
  line width = 2.5pt
]

\tikzstyle{symint} = [
  matlabblue,
  line width = 1.75pt
]

\tikzstyle{genint} = [
   matlaborange,
   line width = 1.75pt
]

\tikzstyle{pod} = [
  matlabgreen,
  line width = 1.75pt
]

\tikzstyle{equi} = [
  dashed
]

\tikzstyle{avg} = [
  dotted
]

\tikzstyle{V} = [
  mark         = *,
  mark options = solid
]

\tikzstyle{VW} = [
  mark         = diamond*,
  mark options = solid
]

\begin{document}


\title{Structured interpolation for multivariate transfer functions of
  quadratic-bilinear systems}
  
\author[$\ast$]{Peter Benner}
\affil[$\ast$]{Max Planck Institute for Dynamics of Complex Technical
  Systems, Sandtorstra{\ss}e 1, 39106 Magdeburg, Germany.
  \email{benner@mpi-magdeburg.mpg.de}, \orcid{0000-0003-3362-4103},
  \authorcr \itshape
  Otto von Guericke University Magdeburg, Faculty of Mathematics,
  Universit{\"a}tsplatz 2, 39106 Magdeburg, Germany.
  \email{peter.benner@ovgu.de}
}

\author[$\dagger$]{Serkan Gugercin}
\affil[$\dagger$]{%
  Department of Mathematics and Division of Computational Modeling and Data
  Analytics, Academy of Data Science, Virginia Tech,
  Blacksburg, VA 24061, USA.\authorcr
  \email{gugercin@vt.edu}, \orcid{0000-0003-4564-5999}
}

\author[$\ddagger$]{Steffen W. R. Werner}
\affil[$\ddagger$]{Courant Institute of Mathematical Sciences, New York
  University, New York, NY 10012, USA.\authorcr
  \email{steffen.werner@nyu.edu}, \orcid{0000-0003-1667-4862}
}
  
\shorttitle{Structured quadratic-bilinear interpolation}
\shortauthor{P. Benner, S. Gugercin, S. W. R. Werner}
\shortdate{2023-04-27}
\shortinstitute{}
  

  
\abstract{%
  High-dimensional/high-fidelity nonlinear dynamical systems appear naturally
  when the goal is to accurately model real-world phenomena. 
  Many physical properties are thereby encoded in the internal differential 
  structure of these resulting large-scale nonlinear systems.
  The  high-dimensionality of the dynamics causes computational bottlenecks,
  especially when these large-scale systems need to be simulated for a variety
  of situations such as different forcing terms.
  This motivates model reduction where the goal is to replace the full-order
  dynamics with accurate reduced-order surrogates. 
  Interpolation-based model reduction has been proven to be an effective tool
  for the construction of cheap-to-evaluate surrogate models that preserve the
  internal structure in the case of weak nonlinearities.  
  In this paper, we consider the construction of multivariate interpolants in
  frequency domain for structured quadratic-bilinear systems.
  We propose definitions for structured variants of the symmetric subsystem and
  generalized transfer functions of quadratic-bilinear systems and provide
  conditions for structure-preserving interpolation by projection.
  The theoretical results are illustrated using two numerical examples including
  the simulation of molecular dynamics in crystal structures.
}

\novelty{%
  We introduce new formulas for structured subsystem transfer functions to
  describe quadratic-bilinear systems with internal differential structures
  in the frequency domain.
  We formulate conditions on projection spaces to enforce structure-preserving
  interpolation for such transfer functions allowing for structure-preserving
  model reduction of quadratic-bilinear systems.
}

\maketitle



\section{Introduction}%
\label{sec:intro}

The accurate modeling of real-world phenomena and processes yields dynamical
systems typically including nonlinearities.
Additionally, these systems often inherit some internal structure from the
underlying physical nature of the considered problem.
An example for such internal structures is the description of internal
system states by second-order time derivatives as it is usually the case in
the modeling of mechanical structures.
Such nonlinear mechanical systems take the form
\begin{equation} \label{eqn:snsys}
  \begin{aligned}
    \tM \ddot{q}(t) & = f\big(q(t), \dot{q}(t), u(t) \big), \\
    y(t) & = \Cp q(t) + \Cv \dot{q}(t),
  \end{aligned}
\end{equation}
with internal states $q(t) \in \R^{\ell}$, describing the system behavior,
the external controls $u(t) \in \R^{m}$ that allow the user to change the
internal behavior, and the quantities of interest $y(t) \in \R^{p}$ that can
be observed from the outside, e.g., by sensor measurements.
Thereby, the first equation in~\cref{eqn:snsys} is a second-order differential
equation with mass (descriptor) matrix $\tM \in \R^{\ell \times \ell}$ and the
nonlinear time evolution function
$f\colon \R^{\ell} \times \R^{\ell} \times \R^{m} \to \R^{\ell}$.
The second, algebraic equation describes the quantities of interest as
linear combination of the states and their first-order derivatives.
Throughout this paper, we assume for any system to have homogeneous initial
conditions.

The Toda lattice is an example of a nonlinear mechanical system of the
form~\cref{eqn:snsys}.
It is used in solid state mechanics to model the movement of
particles in a one-dimensional crystal structure~\cite{Tod67}; see
\Cref{fig:todalattice}.
The nonlinear time evolution function contains exponential terms that
describe the forces between the different particles:
\begin{equation} \label{eqn:todalattice}
  f\big(q(t), \dot{q}(t), u(t) \big) = -\tD \dot{q}(t) -
    \begin{bmatrix} e^{k_{1}(q_{1}(t) - q_{2}(t))} - 1 \\
      e^{k_{2}(q_{2}(t) - q_{3}(t))} - e^{k_{1}(q_{1}(t) - q_{2}(t))} \\
      \vdots \\
      e^{k_{\ell} q_{\ell}(t)} - e^{k_{\ell-1}(q_{\ell-1}(t) - q_{\ell}(t))}
    \end{bmatrix},
\end{equation}
with the positive semidefinite diagonal damping matrix
$\tD \in \R^{\ell \times \ell}$ and the positive stiffness coefficients
$k_{1}, \ldots, k_{\ell}$.
See~\cite[Sec.~1.3.3]{Wer21} for the derivation of the differential model from
the underlying Hamiltonian.

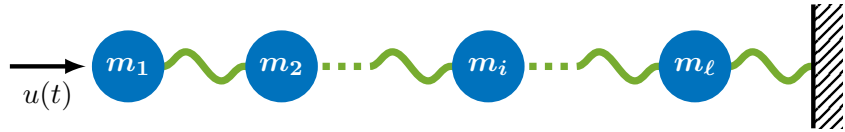
\begin{figure}[t]
  \centering
  \tikzexternalenable%
  \tikzsetnextfilename{todalattice}%
%
%
%

\begin{tikzpicture}[
  x = 2.845em,
  y = 2.845em,
  inner sep = 0,
  outer sep = 0,
  draw      = black]
  
  \colorlet{masscolor}{matlabblue}
  
  \tikzstyle{every node} = [color = black]
  
  \makeatletter
  \tikzset{
    hatch distance/.store in = \hatchdistance,
    hatch distance = .3em,
    hatch thickness/.store in = \hatchthickness,
    hatch thickness = .08em}
    \pgfdeclarepatternformonly[\hatchdistance,\hatchthickness]{north east hatch}
      {\pgfpoint{-.1em}{-.1em}}
      {\pgfpoint{\hatchdistance+.5*\hatchthickness}%
        {\hatchdistance+.5*\hatchthickness}}
      {\pgfpoint{\hatchdistance}{\hatchdistance}}{
        \pgfsetcolor{\tikz@pattern@color}
        \pgfsetlinewidth{\hatchthickness}
        \pgfpathmoveto{\pgfpoint{0em}{0em}}
        \pgfpathlineto{\pgfpoint{\hatchdistance}{\hatchdistance}}
        \pgfusepath{stroke}}
    \makeatother
  
  \tikzstyle{ground} = [
    fill,
    pattern         = north east hatch,
    hatch distance  = .39em,
    hatch thickness = .08em,
    draw            = none,
    minimum width   = 4.2675em,
    minimum height  = 1em,
    rotate          = -90
  ]

  \tikzstyle{mass} = [
    circle,
    minimum size = 2.4em,
    line width   = .08em,
    fill         = masscolor,
    draw         = masscolor,
    text         = white
  ]
  
  \tikzstyle{spring} = [
    line width = .25em,
    draw       = matlabgreen,
    decorate, 
    decoration = {
      snake,
      segment length = 2.45em,
      amplitude      = .5em}
  ]
  
  \node[mass, anchor = west] (mass1) {$\boldsymbol{m_{1}}$};
  
  \begin{pgfonlayer}{background}
    \draw[spring] (mass1.east) -- ++(2.845em, 0) node (tmp) {};
  \end{pgfonlayer}
  
  \node[mass, anchor = west] (mass2) at (tmp) {$\boldsymbol{m_{2}}$};
  
  \draw[
    line width   = .25em,
    dash pattern = on \pgflinewidth off .2845em,
    draw         = matlabgreen
  ] ($(mass2.east) + (.2em, 0)$) -- ++(1.4225em, 0) coordinate (tmp);
  
  \begin{pgfonlayer}{background}
    \draw[spring] ($(tmp) + (.2em, 0)$) -- ++(2.845em, 0) node (tmp) {};
  \end{pgfonlayer}
  
  \node[mass, anchor = west] (massi) at (tmp) {$\boldsymbol{m_{i}}$};
  
  \draw[
    line width   = .25em,
    dash pattern = on \pgflinewidth off .2845em,
    draw         = matlabgreen
  ] ($(massi.east) + (.2em, 0)$) -- ++(1.4225em, 0) coordinate (tmp);
  
  \begin{pgfonlayer}{background}
    \draw[spring] ($(tmp) + (.2em, 0)$) -- ++(2.845em, 0) node (tmp) {};
  \end{pgfonlayer}
  
  \node[mass, anchor = west] (massg) at (tmp) {$\boldsymbol{m_{\ell}}$};
  
  \begin{pgfonlayer}{background}
    \draw[spring] (massg.east) -- ++(2.845em, 0) node (tmp) {};
  \end{pgfonlayer}
  
  \node[ground, anchor = south] (wall) at (tmp) {};
  \draw[line width = .15em] (wall.south west) -- (wall.south east);
  
  \draw[line width = .15em, -latex] ($(mass1.west) + (-2.845em, 0)$)
    -- ($(mass1.west) + (-.25em, 0)$)
    node [midway, below = .5em] {$u(t)$};
\end{tikzpicture}%
  \tikzexternaldisable%

  \caption{Schematic illustration of the Toda lattice with $\ell$
    particles~\cite{Wer21}.
    Atoms in a one-dimensional crystal structure are represented as point
    masses and connected by exponential springs modeling the forces between
    the particles.}
  \label{fig:todalattice}
\end{figure}

In many applications, in particular those involving discretizations of
partial differential equations, the number of differential equations $\ell$
in~\cref{eqn:snsys}, describing the internal system behavior, is large and
increases further with the demand for more accuracy.
However, an increasing amount of differential equations also leads to an
increasing demand for computational resources such as time and memory for
simulations of the models or their use in optimization.
A remedy to this problem is model order reduction, which aims for the
computation of cheap-to-evaluate surrogate models described by significantly
less differential equations, $r \ll \ell$, which approximate the input-to-output
behavior of the original system as
\begin{equation*}
  \lVert y - \hy \rVert \leq \tau \cdot \lVert u \rVert,
\end{equation*}
in some suitable norms for the output of the reduced model $\hy$ and all
admissible inputs $u$.

An established approach for model reduction of general (structured) nonlinear
systems such as~\cref{eqn:snsys} is proper orthogonal decomposition
(POD)~\cite{HinV05, KunV02, WilP02}, in which time simulations are used to
extract information about the system dynamics for the construction of a basis
matrix to project the system states.
Other approaches aim for the extension of balancing-related model reduction
to nonlinear systems using Gramians defined via time simulations as in the
empirical Gramian method~\cite{HimO13, LalMG99, LawMC05} or by new energy
measures~\cite{Sch93, KraGB22} to construct suitable projection matrices.
Nevertheless, a problem arising in the general system case is the approximation
of the nonlinear time evolution function $f$ in~\cref{eqn:snsys} that
circumvents the computationally expensive lifting and truncation of the
low-dimensional state in every time step.
A solution to that are hyperreduction techniques such as the (discrete)
empirical interpolation method
((D)EIM)~\cite{BarMNetal04, ChaS10, DrmG16, Peh20a}, which computes a selection
operator to restrict the evaluation of $f$ to its rows most contributing.
This introduces another layer of approximations and needs explicit access to the
implementation of the original time evolution function $f$.

An alternative to hyperreduction that gained significant popularity in model
reduction in the last decade is quadratic-bilinearization~\cite{Gu11}; in
optimization also known as McCormick relaxation~\cite{McC76}:
For $f$ smooth enough, general nonlinear systems can be rewritten into
quadratic-bilinear form by introducing auxiliary variables and
differential-algebraic equations, which then can be reduced directly using the
classical projection-based model reduction approach.
In the case of~\cref{eqn:snsys}, the corresponding quadratic-bilinear system
retains the internal mechanical structure as
\begin{equation} \label{eqn:sqsys}
  \begin{aligned}
    0 & = M \ddot{q}(t) + D \dot{q}(t) + K q(t) \\
    & \quad{}+{}
      \Hvv\big(\dot{q}(t) \otimes \dot{q}(t)\big) +
      \Hvp\big(\dot{q}(t) \otimes q(t)\big)\\
    &\quad{}+{}
      \Hpv\big(q(t) \otimes \dot{q}(t)\big) +
      \Hpp\big(q(t) \otimes q(t)\big)\\
    & \quad{}-{}
      \sum_{j = 1}^{m} \Nvj{j} \dot{q}(t) u_{j}(t) -
      \sum_{j = 1}^{m} \Npj{j} q(t) u_{j}(t) - \Bu u(t),\\
    y(t) & = \Cp q(t) + \Cv \dot{q}(t),
  \end{aligned}
\end{equation}
where $M, D, K, \Nvj{j}, \Npj{j} \in \R^{n \times n}$, for $j = 1, \ldots, m$,
$\Hvv, \Hvp, \Hpv, \Hpp \in \R^{n \times n^{2}}$, $\Bu \in \R^{n \times m}$,
$\Cp, \Cv \in \R^{p \times n}$, and $\otimes$ denotes the Kronecker product.
Due to the introduction of auxiliary variables, we have that $n \geq \ell$, 
which appears counter-intuitive to the actual task of reducing the number of
internal system states in model reduction.
However, the new nonlinearity structure of~\cref{eqn:sqsys} allows to apply
well-established model reduction techniques without the necessity of the
hyperreduction step for the nonlinearity.
The Toda lattice model (\Cref{fig:todalattice}) can also be rewritten
into~\cref{eqn:sqsys}.
The process of quadratic-bilinearization and the derivation of several
structured quadratic-bilinear formulations for the Toda lattice model are
shown in~\cite[Sec.~6.2]{Wer21}.
Furthermore, we refer the reader to the numerical experiments in
\Cref{subsec:todalattice} for more details.

So far, the literature on the model reduction of quadratic-bilinear systems
mainly considered the case of unstructured, first-order systems of the form
\begin{equation} \label{eqn:qsys}
  \begin{aligned}
    E \dot{x}(t) & = A x(t) + H \big( x(t) \otimes x(t) \big) +
      \sum\limits_{j = 1}^{m} N x(t) u_{j}(t) + B u(t),\\
    y(t) & = C x(t),
  \end{aligned}
\end{equation}
with $E, A, N_{j} \in \R^{n \times n}$, for $j = 1, \ldots, m$,
$H \in \R^{n \times n^{2}}$, $B \in \R^{n \times m}$ and
$C \in \R^{p \times n}$.
Model reduction methods developed for~\cref{eqn:qsys} among others include
the interpolation of multivariate subsystem transfer
functions~\cite{Gu11, AhmBJ16, BenB15, AntBG20},
Volterra series interpolation~\cite{BenGG18, AntBG20},
balanced truncation~\cite{BenG17},
learning models from frequency domain data via the Loewner
framework~\cite{GosA18}
or learning models from time domain data by operator
inference~\cite{PehW16, QiaKPetal20}.
The reformulation of general nonlinear systems into quadratic-bilinear
form~\cref{eqn:qsys} has also been proven to be an effective strategy for
classical nonlinear model reduction methods such as POD~\cite{KraW19}.

In this paper, we extend the idea of quadratic-bilinear subsystem interpolation
to systems with additional internal differential structures such as in the
mechanical case~\cref{eqn:sqsys}.
We propose extensions for the definitions of the first three symmetric
subsystem transfer functions and the first three generalized transfer functions
to the structured system case and then present subspace conditions for
structure-preserving interpolation of these transfer functions.
The effectiveness of resulting model reduction methods based on this
interpolation theory is illustrated on two different structured examples
including the Toda lattice model above.
Parts of the theoretical results presented here were derived in the course of
writing the dissertation of the corresponding author~\cite{Wer21}.

The rest of the paper is organized as follows:
In \Cref{sec:basics}, we recap the ideas of Volterra series expansions and
unstructured quadratic-bilinear systems in the frequency domain.
We present the definitions of structured transfer functions of
quadratic-bilinear systems in \Cref{sec:structtf} and the results on
structure-preserving interpolation in \Cref{sec:interp}.
We employ the interpolation results for model reduction of two structured
numerical examples in \Cref{sec:examples}.
The paper is concluded in \Cref{sec:conclusions}.


\section{Mathematical preliminaries}%
\label{sec:basics}

In this section, we recap the concept of Volterra series for quadratic-bilinear
systems and two resulting transfer function formulations.


\subsection{Volterra series expansions}%
\label{subsec:volterra}

The Volterra series expansion allows to describe the solution of
nonlinear dynamical systems as a series of solutions of coupled linear
systems~\cite{Rug81}.
A common approach to derive the Volterra series expansion is by variational
analysis~\cite{Gu11}.
Let a scaled input signal $\alpha u(t)$, with $\alpha > 0$, be given for the
quadratic-bilinear system~\cref{eqn:qsys} and assume the system state to have
an analytic representation of the form
\begin{equation} \label{eqn:fqstate}
  x(t) = \sum\limits_{k = 1}^{\infty} \alpha^{k} x_{k}(t),
\end{equation}
with a sequence of states $x_{k}(t)$.
Inserting~\cref{eqn:fqstate} into~\cref{eqn:qsys} and extracting the terms
corresponding to the same power of the coefficient $\alpha$ yields the
states $x_{k}(t)$ to be described by cascaded subsystems, which are linear
in their respective unknown state $x_{k}(t)$.
For example, the first three resulting linear subsystems for~\cref{eqn:fqstate}
are given by
\begin{equation} \label{eqn:subqsys}
  \begin{aligned}
    E \dot{x}_{1}(t) & = A x_{1}(t) + B u(t), \\
    E \dot{x}_{2}(t) & = A x_{2}(t) + H
      \big( x_{1}(t) \otimes x_{1}(t) \big) + \sum\limits_{j = 1}^{m}
      N_{j} x_{1}(t) u_{j}(t),\\
    E \dot{x}_{3}(t) & = A x_{3}(t) + H
      \big( x_{1}(t) \otimes x_{2}(t) + x_{2}(t) \otimes x_{1}(t) \big)
      + \sum\limits_{j = 1}^{m} N_{j} x_{2}(t) u_{j}(t).
  \end{aligned}
\end{equation}
Applying the variation-of-constants formula to the subsystems
in~\cref{eqn:subqsys} allows the description of the input-to-output behavior
of~\cref{eqn:qsys} via its Volterra series expansion:
\begin{equation} \label{eqn:fqsysvolterra}
  y(t) = \sum\limits_{k = 1}^{\infty} \int\limits_{0}^{t}
    \int\limits_{0}^{t_{1}} \cdots \int\limits_{0}^{t_{k-1}}
    g_{k}(t_{1}, \ldots, t_{k}) \big( u(t - t_{1}) \otimes \cdots \otimes
    u(t - t_{k}) \big) \dt{t_{k}} \cdots \dt{t_{1}},
\end{equation}
where $g_{k}(t_{1}, \ldots, t_{k})$ are the Volterra kernels of the
corresponding representation.
The kernels used in~\cref{eqn:fqsysvolterra} are of the symmetric
type~\cite{Rug81, Gu11}.
Applying the multivariate Laplace transformation~\cite{Rug81}
to~\cref{eqn:fqsysvolterra} results in an equivalent description of the
quadratic-bilinear system~\cref{eqn:qsys} in the frequency domain by
multivariate transfer functions.


\subsection{Subsystem transfer functions of quadratic-bilinear systems}%
\label{subsec:tf}

In this work, we restrict ourselves to the presentation of the transfer
functions corresponding to the first three coupled linear subsystems for
brevity and practical relevance.
General formulas for arbitrarily high levels of multivariate transfer functions
for~\cref{eqn:qsys} have been developed in~\cite[Sec.~2.3.2]{Wer21}.


\subsubsection{Symmetric subsystem transfer functions}

The symmetric subsystem transfer functions are based on the symmetric Volterra
kernels from~\cite{Rug81}; cf.~\cref{eqn:fqsysvolterra}.
Historically, this is the first transfer function type that has been
investigated for the model reduction of~\cref{eqn:qsys} in~\cite{Gu11}.
Here, the term ``symmetric'' refers to the fact that the transfer function is
invariant with respect to the order of its arguments. 

The first symmetric subsystem transfer function corresponds to the linear
part of~\cref{eqn:qsys} as it can be seen in~\cref{eqn:subqsys} such that
\begin{equation} \label{eqn:qsymtf1}
  \symTF{1}(s_{1}) = C \symtf{1}(s_{1}) = C (s_{1}E - A)^{-1} B,
\end{equation}
with $s_{1} \in \C$.
Thereby, the term $\symtf{1}(s_{1}) \in \C^{n \times m}$ is used in the
following for notational convenience and denotes the input-to-state transition
of the first subsystem.
The second symmetric subsystem transfer function depends on two complex
frequency arguments $s_{1}, s_{2} \in \C$ and is given by
\begin{equation} \label{eqn:qsymtf2}
  \symTF{2}(s_{1}, s_{2}) = C \symtf{2}(s_{1}, s_{2}),
\end{equation}
where the function describing the input-to-state transition on the right-hand
side of~\cref{eqn:qsymtf2} is given by
\begin{equation*}
  \begin{aligned}
    \symtf{2}(s_{1}, s_{2}) & =
        \frac{1}{2} \big((s_{1}+s_{2})E - A \big)^{-1}\\
    & \quad{}\times{}
        \Big( H \big( \symtf{1}(s_{1}) \otimes \symtf{1}(s_{2}) + 
        \symtf{1}(s_{2}) \otimes \symtf{1}(s_{1}) \big)\\
    & \qquad{}+{}
      N \big( I_{m} \otimes \big( \symtf{1}(s_{1}) +
      \symtf{1}(s_{2}) \big) \big) \Big),
  \end{aligned}
\end{equation*}
with $\symtf{1}$ from~\cref{eqn:qsymtf1}, $I_{m}$ denoting the $m$-dimensional
identity matrix and the column concatenation of the bilinear terms as
\begin{equation} \label{eqn:qbilinconcat}
  N = \begin{bmatrix} N_{1} & N_{2} & \ldots & N_{m} \end{bmatrix}.
\end{equation}
Last, the third symmetric subsystem transfer function is defined similarly
to~\cref{eqn:qsymtf2} by
\begin{equation} \label{eqn:qsymtf3}
  \symTF{3}(s_{1}, s_{2}, s_{3}) = C \symtf{3}(s_{1}, s_{2}, s_{3}),
\end{equation}
with $s_{1}, s_{2}, s_{3} \in \C$ and the input-to-state transition
given via
\begin{equation*}
  \begin{aligned}
    \symtf{3}(s_{1}, s_{2}, s_{3}) &
      = \frac{1}{6} \big((s_{1}+s_{2}+s_{3})E - A \big)^{-1}\\
    &  \quad{}\times{}
        \Big( H \big( \symtf{1}(s_{1}) \otimes \symtf{2}(s_{2}, s_{3}) +
        \symtf{1}(s_{2}) \otimes \symtf{2}(s_{1}, s_{3})\\
      & \quad\qquad{}+{}
        \symtf{1}(s_{3}) \otimes \symtf{2}(s_{1}, s_{2}) +
        \symtf{2}(s_{1}, s_{2}) \otimes \symtf{1}(s_{3})\\
      & \quad\qquad{}+{}
        \symtf{2}(s_{1}, s_{3}) \otimes \symtf{1}(s_{2}) +
        \symtf{2}(s_{2}, s_{3}) \otimes \symtf{1}(s_{1}) \big) \\
      & \qquad{}+{}
        N \big( I_{m} \otimes \big(\symtf{2}(s_{1}, s_{2}) + 
        \symtf{2}(s_{1}, s_{3}) + \symtf{2}(s_{2}, s_{3}) \big) \big) \Big),
  \end{aligned}
\end{equation*}
using $\symtf{1}$ from~\cref{eqn:qsymtf1} and $\symtf{2}$
from~\cref{eqn:qsymtf2}.


\subsubsection{Generalized transfer functions}

In contrast to the symmetric case, the generalized transfer functions do not
directly correspond to a Volterra kernel representation.
They have been introduced in~\cite{GosA18} for the extension of the data-driven
Loewner framework to quadratic-bilinear systems and are inspired by the regular
transfer functions of bilinear systems, which only consist of products of the
terms of the dynamical system; see, e.g.,~\cite{BaiS06}.
The formulation given in~\cite{GosA18} for the single-input/single-output (SISO)
system case has been extended to multi-input/multi-output systems
in~\cite[Sec.~2.3.2]{Wer21}.

As in the symmetric case, the first regular transfer function corresponds to
the linear system components, with
\begin{equation} \label{eqn:qgentf1}
  \genTF{1}{(B)}(s_{1}) = C (s_{1} E - A)^{-1} B.
\end{equation}
Also the second regular transfer function is uniquely defined resulting from
one multiplication with the bilinear terms:
\begin{equation} \label{eqn:qgentf2}
  \genTF{2}{(N, (B))}(s_{1},s_{2}) = C (s_{2} E - A)^{-1} N
    \big(I_{m} \otimes (s_{1}E - A)^{-1}B \big).
\end{equation}
For the third level however, two different generalized transfer functions
exist.
The first one is identical to the third regular bilinear transfer function with
\begin{equation} \label{eqn:qgentf31}
  \begin{aligned}
    \genTF{3}{(N, (N, (B)))}(s_{1},s_{2},s_{3}) & = C (s_{3} E - A)^{-1} N
      \Big( I_{m} \otimes \big( (s_{2} E - A)^{-1}N\\ 
    & \quad{}\otimes{}
    (I_{m} \otimes (s_{1}E - A)^{-1}B) \big) \Big);
  \end{aligned}
\end{equation}
see also~\cite{BenGW21a}.
The second one involves the quadratic term and is given by
\begin{equation} \label{eqn:qgentf32}
  \genTF{3}{(H, (B), (B))}(s_{1},s_{2},s_{3}) = C (s_{3} E - A)^{-1} H
    \big((s_{2}E - A)^{-1}B \otimes (s_{1}E - A)^{-1}B \big).
\end{equation}
Note that the index levels of the generalized and symmetric transfer functions
do not coincide since the second symmetric subsystem transfer function does
contain the quadratic term in contrast to the second level generalized transfer
function, which has only the bilinear terms.


\section{Structured transfer functions of quadratic-bilinear systems}%
\label{sec:structtf}

In this section, we extend the transfer function formulations from
\Cref{sec:basics} to the setting of structured quadratic-bilinear systems
starting with the introductory example of quadratic-bilinear mechanical systems.
Based on this motivation, we introduce the formulas for structured symmetric
and generalized transfer functions before we consider the case of
quadratic-bilinear time-delay systems as another example for internal
differential structures at the end of this section.


\subsection{Transfer functions for mechanical systems}%
\label{subsec:mechtf}

In general, any system in second-order form~\cref{eqn:sqsys} can be rewritten
in first-order form~\cref{eqn:qsys} using the concatenated first-order
state $x(t) = \begin{bmatrix} q(t)^{\trans} & \dot{q}(t)^{\trans}
\end{bmatrix}^{\trans}$.
The first-order system matrices are then, for example, given by
\begin{equation} \label{eqn:sfmtx}
  \begin{aligned}
    E & = \begin{bmatrix} I_{n} & 0 \\ 0 & M \end{bmatrix}, &
      A & = \begin{bmatrix} 0 & I_{n} \\ -K & -D \end{bmatrix}, &
      B & = \begin{bmatrix} 0 \\ \Bu \end{bmatrix}, \\
      C & = \begin{bmatrix} \Cp & \Cv \end{bmatrix}, &
      N_{j} & = \begin{bmatrix} 0 & 0 \\ \Npj{j} & \Nvj{j} \end{bmatrix},
  \end{aligned}
\end{equation}
for $j = 1, \ldots, m$, with the quadratic term
\begin{equation} \label{eqn:sfH}
  H = -\begin{bmatrix} 0 & 0 & \ldots & 0 & 0 & 0 & 0 & \ldots & 0 & 0 \\
    H_{\rm{pp}, 1} & H_{\rm{pv}, 1} & \ldots & H_{\rm{pp}, n} & H_{\rm{pv}, n}
    & H_{\rm{vp}, 1} & H_{\rm{vv}, 1} & \ldots & H_{\rm{vp}, n} &
    H_{\rm{vv}, n} \end{bmatrix}.
\end{equation}
Thereby, the matrix blocks in~\cref{eqn:sfH} are $n \times n$ matrix slices
of the second-order quadratic terms in~\cref{eqn:sqsys}, e.g.,
for $\Hpp$ modeling the multiplication of the state with itself we have
\begin{equation*}
  \Hpp = \begin{bmatrix} H_{\rm{pp}, 1} & H_{\rm{pp}, 2} & \ldots &
    H_{\rm{pp}, n} \end{bmatrix},
\end{equation*}
with $H_{\rm{pp}, j} \in \R^{n \times n}$ for all $j = 1, \ldots, n$.

Now, we exploit the block structures of the matrices in~\cref{eqn:sfmtx,eqn:sfH}
to derive the symmetric and generalized transfer functions of~\cref{eqn:sqsys}.
In both transfer function cases, the first level transfer functions correspond
to the linear system case and it can be observed that for~\cref{eqn:sqsys}
it holds that
\begin{equation} \label{eqn:sqtf1}
  \begin{aligned}
    \symTF{1}(s_{1}) = \genTF{1}{(B)}(s_{1})
      & = (\Cp + s_{1} \Cv) \symtf{1}(s_{1})\\
    & = (\Cp + s_{1} \Cv) (s_{1}^{2} M + s_{1} D + K)^{-1} \Bu,
  \end{aligned}
\end{equation}
where $\symtf{1}(s_{1})$ denotes here the input-to-state transition of the
second-order state $q$.

For the next two levels, we concentrate first on the symmetric transfer
function case.
Inserting~\cref{eqn:sfmtx,eqn:sfH} into~\cref{eqn:qsymtf2} yields the
second symmetric subsystem transfer function of~\cref{eqn:sqsys} to be
\begin{equation} \label{eqn:sqsymtf2}
  \begin{aligned}
    \symTF{2}(s_{1}, s_{2}) & = 
      \big( \Cp + (s_{1} + s_{2}) \Cv \big) \symtf{2}(s_{1}, s_{2}) \\
    & = \frac{1}{2} \big( \Cp + (s_{1} + s_{2}) \Cv \big)
      \big( (s_{1} + s_{2})^{2} M + (s_{1} + s_{2}) D + K \big)^{-1} \\
    & \quad{}\times{}
      \Big(-(\Hpp + s_{2} \Hpv + s_{1} \Hvp + s_{1} s_{2} \Hvv)
      \big( \symtf{1}(s_{1}) \otimes \symtf{1}(s_{2}) \big) \\
    & \qquad{}-{}
      (\Hpp + s_{1} \Hpv + s_{2} \Hvp + s_{1} s_{2} \Hvv)
      \big( \symtf{1}(s_{2}) \otimes \symtf{1}(s_{1}) \big)\\
    & \qquad{}+{}
      (\Np + s_{1} \Nv) \big( I_{m} \otimes \symtf{1}(s_{1}) \big)\\
    & \qquad{}+{}
      (\Np + s_{2} \Nv) \big( I_{m} \otimes \symtf{1}(s_{2}) \big)
      \Big),
  \end{aligned}
\end{equation}
where $\symtf{2}$ denotes the input-to-state transition of the
second subsystem, $\symtf{1}$ is the input-to-state transition from
the first subsystem in~\cref{eqn:sqtf1}, and the bilinear terms are concatenated
as
\begin{equation} \label{eqn:sqbilinconcat}
  \begin{aligned}
    \Np & = \begin{bmatrix} \Npj{1} &  \ldots & \Npj{m} \end{bmatrix} &
    \text{and} &&
    \Nv & = \begin{bmatrix} \Nvj{1} & \ldots & \Nvj{m} \end{bmatrix}.
  \end{aligned}
\end{equation}
Similarly, inserting the block matrices~\cref{eqn:sfmtx,eqn:sfH} into
the unstructured third symmetric subsystem transfer function~\cref{eqn:qsymtf3}
leads to the structured third symmetric subsystem transfer function
representation for~\cref{eqn:sqsys}.
Due to the complexity of the resulting formula, we will not write it out
here explicitly but outline some of its features.
The third subsystem transfer function of~\cref{eqn:qsys} has a similar structure
to~\cref{eqn:qsymtf3,eqn:sqsymtf2} with linear combinations of the quadratic
and bilinear terms multiplied with the previous input-to-state transition terms
$\symtf{1}$ and $\symtf{2}$ from~\cref{eqn:sqtf1,eqn:sqsymtf2}.
Due to the occurrence of the sum of the frequency arguments in the first term
of~\cref{eqn:sqsymtf2}, this translates into the frequency dependence of the
second-order quadratic and bilinear terms such that terms of the forms
\begin{equation} \label{eqn:sqsymtf3Tmp}
  \begin{aligned}
    & -\big( \Hpp + s_{3} \Hpv + (s_{1} + s_{2}) \Hvp +
      (s_{1} + s_{2})s_{3} \Hvv \big)
      \big( \symtf{2}(s_{1}, s_{2}) \otimes \symtf{1}(s_{3}) \big)\\
    & \text{and}\quad \big(\Np + (s_{1} + s_{2}) \Nv \big)
      \big( I_{m} \otimes \symtf{2}(s_{1}, s_{2}) \big)
  \end{aligned}
\end{equation}
appear.
For more details, we refer the reader to the next section, which contains the
definitions of the transfer function formulas for general structures.

For the generalized transfer functions, one can observe that the second level
transfer function resembles the corresponding bilinear regular transfer
function, for which the structured system case has been developed
in~\cite{BenGW21a}.
The resulting transfer function for~\cref{eqn:sqsys} is thereby given as
\begin{equation} \label{eqn:sqgentf2}
  \begin{aligned}
    \genTF{2}{(N, (B))}(s_{1}, s_{2}) & =
      (\Cp + s_{2} \Cv) (s_{2}^{2} M + s_{2} D + K)^{-1} (\Np + s_{1} \Nv)\\
    & \quad{}\times{}
      \big( I_{m} \otimes (s_{1}^{2} M + s_{1} D + K)^{-1}
      \Bu \big),
  \end{aligned}
\end{equation}
where the bilinear terms are concatenated as in~\cref{eqn:sqbilinconcat}.
Similarly, the purely bilinear third level generalized transfer function
of~\cref{eqn:sqsys} is given by
\begin{equation} \label{eqn:sqgentf31}
  \begin{aligned}
    \genTF{3}{(N, (N, (B)))} & = (\Cp + s_{3} \Cv)
      (s_{3}^{2} M + s_{3} D + K)^{-1} (\Np + s_{2} \Nv)\\
    & \quad{}\times{}
      \Big( I_{m} \otimes (s_{2}^{2} M + s_{2} D + K)^{-1} (\Np + s_{1} \Nv) \\
    & \qquad{}\times{}
      \big( I_{m} \otimes (s_{1}^{2} M + s_{1} D + K)^{-1} \Bu \big) \Big).
  \end{aligned}
\end{equation}
On the other hand, the third level generalized transfer function
of~\cref{eqn:sqsys} involving the quadratic term can be derived by
inserting~\cref{eqn:sfmtx} into~\cref{eqn:qgentf32}, which yields
\begin{equation} \label{eqn:sqgentf32}
  \begin{aligned}
    \genTF{3}{(H, (B), (B))} (s_{1}, s_{2}, s_{3}) & =
      -(\Cp + s_{3} \Cv) (s_{3}^{2} M + s_{3} D + K)^{-1}\\
    & \quad{}\times{}
      (\Hpp + s_{1} \Hpv + s_{2} \Hvp + s_{1} s_{2} \Hvv)\\
    & \quad{}\times{}
      \big( (s_{2}^{2} M + s_{2} D + K)^{-1} \Bu \otimes 
      (s_{1}^{2} M + s_{1} D + K)^{-1} \Bu \big).
  \end{aligned}
\end{equation}

Overall, and similar to the linear and bilinear cases, we can observe the
occurrence of the same terms describing the linear, bilinear or quadratic
dynamics in the different transfer functions by means of the given system
matrices and the complex variables $s_{1}, s_{2}, s_{3}$.
This motivates the definitions for the general structured framework in the
upcoming section.


\subsection{Structured transfer function formulas for quadratic-bilinear
  systems}

Before deriving structured quadratic-bilinear transfer functions, 
we briefly recall the ideas from~\cite{BeaG09}, which considered the structured
transfer functions for linear dynamical systems in which
frequency-dependent equations are used for describing the dynamics. 
First, consider linear first-order (unstructured) systems
with the time domain representation
\begin{equation}  \label{eqn:linunst}
  \begin{aligned}
    E \dot{x}(t) - A x(t) & = B u(t), &
    y(t) = C x(t).
  \end{aligned}
\end{equation}
By taking the Laplace transform of~\cref{eqn:linunst}, the dynamical system
in~\cref{eqn:linunst} can be equivalently described in the frequency domain as
\begin{equation} \label{eqn:tflinunst}
  \begin{aligned}
    (sE - A) X(s) & = B U(s), &
    Y(s) = C X(s),
  \end{aligned}
\end{equation}
with $s \in \C$, where $U(s), X(s), Y(s)$ are the Laplace transforms of
the inputs $u(t)$, states  $x(t)$, and outputs $y(t)$, respectively.
Now consider a linear dynamical system with a second-order structure with the
time domain representation 
\begin{equation} \label{eqn:linstr}
  \begin{aligned}
    M \ddot{x}(t) + D \dot{x}(t) + K x(t) & = \Bu u(t), &
    y(t) = \Cp x(t) + \Cv \dot{x}(t).
  \end{aligned}
\end{equation}
As in the unstructured case, taking the Laplace transform of~\cref{eqn:linstr}
yields the representation in the frequency domain 
\begin{equation} \label{eqn:tflinst}
  \begin{aligned}
    (s^{2} M + s D + K) X(s) & = \Bu U(s), &
    Y(s) = (\Cp + s \Cv) X(s).
  \end{aligned}
\end{equation}
Observe that in both cases of~\cref{eqn:tflinunst} and~\cref{eqn:tflinst},  the
system states in the  frequency domain can be described as solution of
frequency-dependent linear systems of equations of the form
\begin{equation} \label{eqn:linstruct}
  \begin{aligned}
        \cK(s) X(s) & = \cB(s) U(s), & Y(s) & = \cC(s),
  \end{aligned}
\end{equation}
where the matrix-valued functions
$\cK\colon \C \to \C^{n \times n}$,
$\cB\colon \C \to \C^{n \times m}$ and
$\cC\colon \C \to \C^{p \times m}$ describe the linear dynamics, and
the input and output behavior of the system.
In particular, we recover~\cref{eqn:tflinunst} by setting 
$\cK(s) = sE - A$, $\cB(s) = B$, and  $\cC(s) = B$.
Similarly, we recover~\cref{eqn:tflinst} by setting
$\cK(s) = s^2M + s D +  K$, $\cB(s) = \Bu$, and  $\cC(s) = \Cp + s \Cv$.
We refer the reader to~\cite{BeaG09} for further examples of structured
dynamics that fit into the general framework of~\cref{eqn:linstruct}.
Then, for every $s \in \C$ for which $\cK(s)$ in~\cref{eqn:linstruct} is
invertible, the transfer function of the underlying linear dynamical system
is given by
\begin{equation} \label{eqn:lintf}
  G(s) = \cC(s) \cK(s)^{-1} \cB(s).
\end{equation}
For the extension to structured bilinear systems in~\cite{BenGW21a}, a new
frequency-dependent function $\cN\colon \C \to \C^{n \times nm}$ was
introduced, modeling the effect of the bilinear terms, where
\begin{equation*}
  \cN(s) = \begin{bmatrix} \cN_{1}(s) & \ldots & \cN_{m}(s) \end{bmatrix},
\end{equation*}
with $\cN_{j}\colon \C \to \C^{n \times n}$ for all $j = 1, \ldots, m$.
In this manuscript, we further extend the structured transfer function
framework to the  quadratic-bilinear case, which appears ubiquitously in
prominent applications as we briefly discussed in \Cref{sec:intro}.

First, we consider the symmetric transfer function case.
Inspired by~\cref{eqn:qsymtf1,eqn:qsymtf2,eqn:qsymtf3}, from the unstructured
first-order case, and~\cref{eqn:sqtf1,eqn:sqsymtf2,eqn:sqsymtf3Tmp}, we
introduce the following definition for the structured symmetric subsystem
transfer functions.

\begin{definition}\allowdisplaybreaks%
  \label{def:symTF}
  Given matrix-valued functions of the form
  $\cC\colon \C \to \C^{p \times m}$,
  $\cK\colon \C \to \C^{n \times n}$,
  $\cB\colon \C \to \C^{n \times m}$,
  $\cN\colon \C \to \C^{n \times nm}$,
  $\cH\colon \C \times \C \to \C^{n \times n^{2}}$,
  for which there exists an $s \in \C$ at which they can be evaluated and
  $\cK(s)$ is invertible.
  The first three \emph{structured symmetric subsystem transfer functions} are
  defined as
  \begin{align*}
    \symTF{1}(s_{1}) & = \cC(s_{1}) \symtf{1}(s_{1}), \\
    \symTF{2}(s_{1}, s_{2}) & = \cC(s_{1} + s_{2})
      \symtf{2}(s_{1}, s_{2}), \\
    \symTF{3}(s_{1}, s_{2}, s_{3}) & = \cC(s_{1} + s_{2} + s_{3})
      \symtf{3}(s_{1}, s_{2}, s_{3}),
  \end{align*}
  where the input-to-state transitions are recursively given by
  \begin{align*}
    \symtf{1}(s_{1}) & = \cK(s_{1})^{-1} \cB(s_{1}), \\
    \symtf{2}(s_{1}, s_{2}) & = \frac{1}{2}
      \cK(s_{1} + s_{2})^{-1}
      \Big( \cH(s_{1}, s_{2}) \big( \symtf{1}(s_{1}) \otimes
      \symtf{1}(s_{2}) \big) \\
    & \quad{}+{}
      \cH(s_{2}, s_{1}) \big( \symtf{1}(s_{2}) \otimes \symtf{1}(s_{1})
      \big) +
      \cN(s_{1}) \big( I_{m} \otimes \symtf{1}(s_{1}) \big) \\
    & \quad{}+{}
      \cN(s_{2}) \big( I_{m} \otimes \symtf{1}(s_{2}) \big) \Big),\\
    \symtf{3}(s_{1}, s_{2}, s_{3}) & = \frac{1}{6}
      \cK(s_{1} + s_{2} + s_{3})^{-1} \Big(
      \cH(s_{1} + s_{2}, s_{3}) \big( \symtf{2}(s_{1}, s_{2}) \otimes
      \symtf{1}(s_{3}) \big) \\
    & \quad{}+{}
      \cH(s_{1} + s_{3}, s_{2}) \big( \symtf{2}(s_{1}, s_{3}) \otimes
      \symtf{1}(s_{2}) \big) \\
    & \quad{}+{}
      \cH(s_{2} + s_{3}, s_{1}) \big( \symtf{2}(s_{2}, s_{3}) \otimes
      \symtf{1}(s_{1}) \big) \\
    & \quad{}+{}
      \cH(s_{1}, s_{2} + s_{3}) \big( \symtf{1}(s_{1}) \otimes
      \symtf{2}(s_{2}, s_{3}) \big)\\
    & \quad{}+{}
      \cH(s_{2}, s_{1} + s_{3}) \big( \symtf{1}(s_{2}) \otimes
      \symtf{2}(s_{1}, s_{3}) \big) \\
    & \quad{}+{}
      \cH(s_{3}, s_{1} + s_{2}) \big( \symtf{1}(s_{3}) \otimes
      \symtf{2}(s_{1}, s_{2}) \big)\\
    & \quad{}+{}
      \cN(s_{1} + s_{2}) \big( I_{m} \otimes \symtf{2}(s_{1}, s_{2}) \big)\\
    & \quad{}+{}
      \cN(s_{1} + s_{3}) \big( I_{m} \otimes \symtf{2}(s_{1}, s_{3}) \big)\\
    & \quad{}+{}
      \cN(s_{2} + s_{3}) \big( I_{m} \otimes \symtf{2}(s_{2}, s_{3}) \big)
      \Big).
  \end{align*}
\end{definition}

Similarly, we give a definition for the structured variant of the generalized
transfer functions inspired by the first-order
case~\cref{eqn:qgentf1,eqn:qgentf2,eqn:qgentf31,eqn:qgentf32},
and second-order case~\cref{eqn:sqtf1,eqn:sqgentf2,eqn:sqgentf31,eqn:sqgentf32}
in the following.

\begin{definition}\allowdisplaybreaks%
  \label{def:genTF}
  Given matrix-valued functions of the form
  $\cC\colon \C \to \C^{p \times m}$,
  $\cK\colon \C \to \C^{n \times n}$,
  $\cB\colon \C \to \C^{n \times m}$,
  $\cN\colon \C \to \C^{n \times nm}$,
  $\cH\colon \C \times \C \to \C^{n \times n^{2}}$,
  for which there exists an $s \in \C$ at which they can be evaluated and
  $\cK(s)$ is invertible.
  The first three levels of \emph{structured generalized transfer functions} are
  defined as
  \begin{align*}
    \genTF{1}{(B)}(s_{1}) & = \cC(s_{1}) \cK(s_{1})^{-1} \cB(s_{1}), \\
    \genTF{2}{(N, (B))}(s_{1}, s_{2}) & = \cC(s_{2})
      \cK(s_{2})^{-1} \cN(s_{1}) \big( I_{m} \otimes \cK(s_{1})^{-1}
      \cB(s_{1}) \big), \\
    \genTF{3}{(N, (N, (B)))}(s_{1}, s_{2}, s_{3}) & =
      \cC(s_{3}) \cK(s_{3})^{-1} \cN(s_{2})
      \Big( I_{m} \otimes \cK(s_{2})^{-1} \cN(s_{1})\\
    & \quad{}\times{}
      \big( I_{m} \otimes \cK(s_{1})^{-1} \cB(s_{1}) \big) \Big),\\
    \genTF{3}{(H, (B), (B))}(s_{1}, s_{2}, s_{3}) & =
      \cC(s_{3}) \cK(s_{3})^{-1} \cH(s_{2}, s_{1})
      \big( \cK(s_{2})^{-1} \cB(s_{2}) \otimes
      \cK(s_{1})^{-1} \cB(s_{1}) \big).
  \end{align*}
\end{definition}

In~\cite{BenG21}, a simplified variant of the generalized transfer
functions~\cref{eqn:qgentf1,eqn:qgentf2,eqn:qgentf31,eqn:qgentf32} has been
extended to systems with polynomial nonlinearities.
Based on the structured definitions above, these simplified generalized
transfer functions have then been extended to the structured case
in~\cite{GoyPB23}.

Note that in both \Cref{def:symTF,def:genTF}, the new matrix-valued function
$\cH\colon \C \times \C \to \C^{n \times n^{2}}$ results from the quadratic
terms in the time domain.

Both examples of internal system structures considered so far can be
represented in the new structured transfer function framework.
Unstructured first-order systems of the form~\cref{eqn:qsys} are given by
\begin{equation*}
  \begin{aligned}
    \cC(s) & = C, &
    \cK(s) & = s E - A, &
    \cB(s) & = B, &
    \cN(s) & = N, &
    \cH(s_{1}, s_{2}) & = H,
  \end{aligned}
\end{equation*}
where the bilinear terms are concatenated as in~\cref{eqn:qbilinconcat}.
For the second-order system of the form~\cref{eqn:sqsys}, the symmetric and
generalized transfer functions given in \Cref{subsec:mechtf} can be recovered
from \Cref{def:symTF,def:genTF} using
\begin{equation*}
  \begin{aligned}
    \cC(s) & = \Cp + s \Cv, \\
    \cK(s) & = s^{2} M + s D + K, \\
    \cB(s) & = \Bu, \\
    \cN(s) & = \Np + s \Nv, \\
    \cH(s_{1}, s_{2}) & = -(\Hpp + s_{2} \Hpv + s_{1} \Hvp + s_{1} s_{2} \Hvv),
  \end{aligned}
\end{equation*}
with the bilinear terms concatenated as in~\cref{eqn:sqbilinconcat}.
For the definition of higher level structured transfer functions for
quadratic-bilinear systems see~\cite[Sec.~6.3]{Wer21}.


\subsection{Another example structure: Quadratic-bilinear time-delay systems}%
\label{subsec:tdtf}

Before we consider interpolation of the structured transfer functions
in \Cref{def:symTF,def:genTF}, we present another example for internal system
structures that are covered by the new structured transfer function framework.
Quadratic-bilinear systems with constant time delays in the linear dynamic
components can be written as
\begin{equation} \label{eqn:tdqsys}
  \begin{aligned}
    E \dot{x}(t) & = \sum\limits_{k = 1}^{\ell} A_{k} x(t - \tau_{k})
      + H\big( x(t) \otimes x(t) \big)
      + \sum\limits_{j = 1}^{m} N_{j} x(t) u_{j}(t)
      + B u(t),\\
    y(t) & = C x(t),
  \end{aligned}
\end{equation}
with the matrices $A_{k} \in \R^{n \times n}$ describing the effect of state
delayed by $\tau_{k} \in \R_{\geq 0}$, for all $k = 1, \ldots, \ell$, and
the remaining system matrices as defined in~\cref{eqn:qsys}.
Following the variational analyses from~\cref{eqn:subqsys}, we observe that the
time-delay structure only affects the terms with the linear dynamics.
Therefore, the structured transfer functions for~\cref{eqn:tdqsys} are
given by using the matrix-valued functions
\begin{equation*}
  \begin{aligned}
    \cC(s) & = C, &
    \cK(s) & = s E - \sum\limits_{k = 1}^{\ell} A_{k} e^{-\tau_{k} s}, &
    \cB(s) & = B, &
    \cN(s) & = N, &
    \cH(s_{1}, s_{2}) & = H
  \end{aligned}
\end{equation*}
in \Cref{def:symTF,def:genTF}.
This is in accordance to the results for bilinear time-delay systems obtained
in~\cite{GosPBetal19,BenGW21a}.


\section{Structured transfer function interpolation}%
\label{sec:interp}

In this section, we present results on the construction of structured
interpolants for the symmetric or generalized transfer functions from the
previous section.


\subsection{Structure-preserving model reduction via projection}

For the construction of interpolating reduced-order models, we will use the
projection approach in this work.
Thereby, two constant basis matrices $V, W \in \C^{n \times r}$ are constructed,
which allow the computation of the reduced-order quantities via multiplication
with the original system matrices.
Given the full-order matrix-valued functions
$\cC\colon \C \to \C^{p \times n}$,
$\cK\colon \C \to \C^{n \times n}$,
$\cB\colon \C \to \C^{n \times m}$,
$\cN\colon \C \to \C^{n \times nm}$,
$\cH\colon \C \times \C \to \C^{n \times n^{2}}$, that describe a structured
quadratic-bilinear system in the frequency domain, reduced-order model
quantities are computed by
\begin{equation} \label{eqn:projection}
  \begin{aligned}
    \hcC(s) & = \cC V, &
    \hcK(s) & = W^{\herm} \cK(s) V, \\
    \hcB(s) & = W^{\herm} \cB(s), &
    \hcN(s) & = W^{\herm} \cN(s) (I_{m} \otimes V) \\
    \text{and} & &
    \hcH(s_{1}, s_{2}) & = W^{\herm} \cH(s_{1}, s_{2}) (V \otimes V),
  \end{aligned}
\end{equation}
where $W^{\herm} := \overline{W}^{\trans}$ denotes the conjugate transpose of
the matrix $W$.
The Kronecker product in the multiplication with the concatenation of the
bilinear terms in~\cref{eqn:projection} boils down to the multiplication of each
single bilinear term with the two basis matrices as
\begin{equation*}
  \hcN(s) = \begin{bmatrix} \hcN_{1}(s) & \ldots & \hcN_{m}(s) \end{bmatrix}
    = \begin{bmatrix} W^{\herm} \cN_{1}(s) V & \ldots &
    W^{\herm} \cN_{m}(s) V \end{bmatrix}.
\end{equation*}
Moreover, the Kronecker product of the basis matrix $V$ for the reduction of the
quadratic term in~\cref{eqn:projection} can be implemented efficiently without
explicitly forming $V \otimes V$, using techniques from tensor algebra; see,
e.g.,~\cite{BenB15, BenG17, Wer21}.

Model reduction by projection preserves internal structures by construction.
Any ma\-trix-valued function can be decomposed into frequency-affine form,
e.g., in the case of the term describing the linear dynamics, it can be
written as
\begin{equation} \label{eqn:freqaffine}
  \cK(s) = \sum\limits_{j = 1}^{n_{\cK}} h_{j}(s) \cK_{j},
\end{equation}
with $n_{\cK} \in \N$, some frequency-dependent scalar functions
$h_{j}\colon \C \to \C$ and constant matrices $\cK_{j} \in \C^{n \times n}$,
for all $j = 1, \ldots, n_{\cK}$.
The reduced-order matrix-valued function is then given by
\begin{equation} \label{eqn:freqaffineROM}
  \hcK(s) = W^{\herm} \cK(s) V = \sum\limits_{j = 1}^{n_{\cK}} h_{j}(s)
    W^{\herm} \cK_{j} V =
    \sum\limits_{j = 1}^{n_{\cK}} h_{j}(s) \hcK_{j},
\end{equation}
with the reduced-order constant matrices $\hcK_{j} \in \C^{r \times r}$,
for all $j = 1, \ldots, n_{\cK}$.
The frequency-dependent scalar functions in~\cref{eqn:freqaffineROM} are the
same as in~\cref{eqn:freqaffine}, i.e., the internal structure is preserved
and the reduced-order matrices replace their high-dimensional counterparts
from the original system to describe the reduced-order model.

To illustrate the computation of reduced-order quadratic-bilinear systems
via projection, we consider the two motivational differential
structures from the previous section.
In the case of second-order quadratic-bilinear systems~\cref{eqn:sqsys},
reduced-order systems are computed as
\begin{equation*}
  \begin{aligned}
    \hcC(s) & = \Cp V + s \Cv V, &
    \hcK(s) & = s^{2} W^{\herm} M V + W^{\herm} D V + W^{\herm} K V, \\
    \hcB(s) & = W^{\herm} \Bu, &
    \hcN(s) & = W^{\herm} \Np (I_{m} \otimes V) +
      s W^{\herm} \Nv (I_{m} \otimes V),
  \end{aligned}
\end{equation*}
with the reduced quadratic terms given by
\begin{equation*}
  \begin{aligned}
    \hcH(s_{1}, s_{2}) & = -(W^{\herm} \Hpp (V \otimes V) +
      s_{2} W^{\herm} \Hpv (V \otimes V)\\
    & \quad{}+{}
      s_{1} W^{\herm} \Hvp (V \otimes V)
      + s_{1} s_{2} W^{\herm} \Hvv (V \otimes V)).
  \end{aligned}
\end{equation*}
Evaluating the matrix products yields the reduced-order matrices that
represent the reduced-order system in the same structure as the original
system~\cref{eqn:sqsys}.
Similarly, for the quadratic-bilinear time-delay system~\cref{eqn:tdqsys},
reduced-order systems are computed via
\begin{equation*}
  \begin{aligned}
    \hcC(s) & = C V, &
    \hcK(s) & = s W^{\herm} E V - \sum\limits_{k = 1}^{\ell} W^{\herm} A_{k} V
      e^{-\tau_{k}s}, \\
    \hcB(s) & = W^{\herm} B, &
    \hcN(s) & = W^{\herm} N (I_{m} \otimes V) \\
    \text{and} &&
    \hcH(s_{1}, s_{2}) & = W^{\herm} H (V \otimes V).
  \end{aligned}
\end{equation*}
As in the second-order system case, evaluating the matrix products allows
to replace the original, high-dimensional system matrices in~\cref{eqn:tdqsys}
by the reduced ones to describe the reduced-order system using the same
structure.

The essential question of projection-based model order reduction is the
construction of the basis matrices $V$ and $W$.
In the following, conditions are derived to enforce interpolation of
the original symmetric or generalized transfer functions by the corresponding
transfer functions given via the reduced matrix-valued
functions~\cref{eqn:projection}.


\subsection{Interpolating symmetric transfer functions}

In this section, we consider the interpolation of the structured symmetric
subsystem transfer functions from \Cref{def:symTF}.
The following proposition states some first results for the general
interpolation of the first two symmetric subsystem transfer functions at
different frequency points.

\begin{proposition}[{\cite[Cor.~6.3]{Wer21}}]%
  \label{prp:symV}
  Let $G$ be a quadratic-bilinear system, described by its symmetric subsystem
  transfer functions $\symTF{k}$ from \Cref{def:symTF}, and $\hG$ the
  reduced-order quadratic-bilinear system constructed by~\cref{eqn:projection},
  with its reduced-order symmetric subsystem transfer functions~$\hsymTF{k}$.
  Also, let $\sigma_{1}, \sigma_{2} \in \C$ be interpolation points
  such that the matrix-valued functions $\cC, \cK, \cB, \cN, \cH$ and
  $\cK(.)^{-1}$ are defined in these points and their sum.
  Construct the basis matrix $V$ by
  \begin{equation*}
    \begin{aligned}
      V_{1, 1} & = \cK(\sigma_{1})^{-1} \cB(\sigma_{1}), \\
      V_{1, 2} & = \cK(\sigma_{2})^{-1} \cB(\sigma_{2}), \\
      V_{2} & = \cK(\sigma_{1} + \sigma_{2})^{-1} \big(
        \cH(\sigma_{1}, \sigma_{2}) (V_{1, 1} \otimes V_{1, 2}) +
        \cH(\sigma_{2}, \sigma_{1}) (V_{1, 2} \otimes V_{1, 1})\\
      & \quad{}+{}
        \cN(\sigma_{1}) (I_{m} \otimes V_{1, 1}) +
        \cN(\sigma_{2}) (I_{m} \otimes V_{1, 2}) \big), \\
      \mspan(V) & \supseteq \mspan\left(\begin{bmatrix} V_{1, 1} & V_{1, 2} &
        V_{2} \end{bmatrix} \right),
    \end{aligned}
  \end{equation*}
  and let $W$ be an arbitrary full-rank matrix of appropriate dimensions.
  Then, the symmetric subsystem transfer functions of $\hG$ interpolate those
  of $G$ in the following way:
  \begin{equation*}
    \begin{aligned}
      \symTF{1}(\sigma_{1}) & = \hsymTF{1} (\sigma_{1}),\\
      \symTF{1}(\sigma_{2}) & = \hsymTF{1} (\sigma_{2}),\\
      \symTF{2}(\sigma_{1}, \sigma_{2}) & = \hsymTF{1}(\sigma_{1}, \sigma_{2}).
    \end{aligned}
  \end{equation*}
\end{proposition}

Note that using~\cite[Thm.~6.2]{Wer21}, the basis matrix $V$ in \Cref{prp:symV}
can be extended such that the third symmetric subsystem transfer function
is interpolated as well.
In practice however, this result is barely used due to the exponential increase
of terms to evaluate for the construction of the projection space and the
corresponding computational complexity.
Therefore, it is omitted here.

The next proposition considers a similar interpolation result as in
\Cref{prp:symV} by setting conditions on the second basis matrix as well.

\begin{proposition}[{\cite[Lem.~6.4]{Wer21}}]%
  \label{prp:symVW}
  Given the same assumptions as in \Cref{prp:symV},
  let the matrices $V_{1, 1}$ and $V_{1, 2}$ be as in \Cref{prp:symV}.
  Construct the two basis matrices such that
  \begin{equation*}
    \begin{aligned}
      \mspan(V) & \supseteq \mspan\left( \begin{bmatrix} V_{1, 1} & V_{1, 2}
        \end{bmatrix} \right), \\
      \mspan(W) & \supseteq \mspan\left( \cK(\sigma_{1} + \sigma_{2})^{-\herm}
        \cC(\sigma_{1} + \sigma_{2})^{\herm} \right),
    \end{aligned}
  \end{equation*}
  and let $V$ and $W$ be of the same dimension.
  Then, the symmetric subsystem transfer functions of $\hG$ interpolate those
  of $G$ in the following way:
  \begin{equation*}
    \begin{aligned}
      \symTF{1}(\sigma_{1}) & = \hsymTF{1}(\sigma_{1}), \\
      \symTF{1}(\sigma_{2}) & = \hsymTF{1}(\sigma_{2}), \\
      \symTF{1}(\sigma_{1} + \sigma_{2}) & =
        \hsymTF{1}(\sigma_{1} + \sigma_{2}), \\
      \symTF{2}(\sigma_{1}, \sigma_{2}) & = 
        \hsymTF{2}(\sigma_{1}, \sigma_{2}).
    \end{aligned}
  \end{equation*}
\end{proposition}

The result of \Cref{prp:symVW} allows to enforce the same and more interpolation
conditions than in \Cref{prp:symV} in an implicit way using the second basis
matrix.
This reduces the computational complexity of the construction of the
basis matrices, since no nonlinear terms are involved, and allows to match
more interpolation conditions with smaller reduced-order models.

The choice of interpolation points for the different subsystem levels is crucial
for the quality of the computed reduced-order model.
Good or even optimal choices of interpolation points are currently unknown.
However, an advantageous choice to minimize the amount of basis contributions
necessary for the interpolation of higher level subsystem transfer functions
in the symmetric case is
$\sigma_{1} = \sigma_{2} = \sigma_{3} = \sigma$.
The following theorem states the interpolation conditions for this particular
selection of interpolation points and also gives conditions for the
interpolation of the third subsystem transfer function.

\begin{theorem}%
  \label{thm:symVW}
  Let $G$ be a quadratic-bilinear system, described by its symmetric subsystem
  transfer functions $\symTF{k}$ as in \Cref{def:symTF}, and $\hG$ the
  reduced-order quadratic-bilinear system constructed by~\cref{eqn:projection},
  with its reduced-order symmetric subsystem transfer functions~$\hsymTF{k}$.
  Also, let $\sigma \in \C$ be an interpolation point such that the
  matrix-valued functions $\cC, \cK, \cB, \cN, \cH$ and
  $\cK(.)^{-1}$ are defined at $\sigma$ as well as at $2\sigma$ and $3\sigma$.
  Construct the matrices
  \begin{equation*}
    \begin{aligned}
      V_{1} & = \cK(\sigma)^{-1} \cB(\sigma), \\
      V_{2} & = \cK(2 \sigma)^{-1} \big(\cH(\sigma, \sigma)
        (V_{1} \otimes V_{1})  + \cN(\sigma) (I_{m} \otimes V_{1})\big),\\
      V_{3} & = \cK(3 \sigma)^{-1} \big(\cH(2\sigma, \sigma)
        (V_{2} \otimes V_{1}) + \cH(\sigma, 2\sigma)
        (V_{1} \otimes V_{2}) + \cN(2\sigma) (I_{m} \otimes V_{2})\big),
    \end{aligned}
  \end{equation*}
  and
  \begin{equation*}
    \begin{aligned}
      W_{1} & = \cK(2 \sigma)^{-\herm} \cC(2 \sigma)^{\herm},\\
      W_{2} & = \cK(3 \sigma)^{-\herm} \cC(3 \sigma)^{\herm}.
    \end{aligned}
  \end{equation*}
  Then, the following statements hold true:
  \begin{enumerate}[label = (\alph*)]
    \item If the basis matrix $V$ is such that
      \begin{equation*}
        \begin{aligned}
          \mspan(V) \supseteq \mspan \left( \begin{bmatrix} V_{1} & V_{2} &
            V_{3} \end{bmatrix} \right),
        \end{aligned}
      \end{equation*}
      and $W$ is full-rank and of the same  dimension as $V$, then the symmetric
      transfer functions of $\hG$ interpolate those of $G$ in the following
      way:
      \begin{equation*}
        \begin{aligned}
          \symTF{1}(\sigma) & = \hsymTF{1}(\sigma), \\
          \symTF{2}(\sigma, \sigma) & = \hsymTF{2}(\sigma, \sigma), \\
          \symTF{3}(\sigma, \sigma, \sigma) &
            = \hsymTF{3}(\sigma, \sigma, \sigma).
        \end{aligned}
      \end{equation*}
    \item If the basis matrices $V$ and $W$ are such that
      \begin{equation*}
        \begin{aligned}
          \mspan(V) & \supseteq \mspan \left( V_{1} \right) &
          \text{and} &&
          \mspan(W) & \supseteq \mspan \left( W_{1} \right),
        \end{aligned}
      \end{equation*}
      and have the same dimension, then the symmetric
      transfer functions of $\hG$ interpolate those of $G$ in the following
      way:
      \begin{equation*}
        \begin{aligned}
          \symTF{1}(\sigma) & = \hsymTF{1}(\sigma), \\
          \symTF{1}(2 \sigma) & = \hsymTF{1}(2 \sigma), \\
          \symTF{2}(\sigma, \sigma) & = \hsymTF{2}(\sigma, \sigma).
        \end{aligned}
      \end{equation*}
    \item If the basis matrices $V$ and $W$ are such that
      \begin{equation*}
        \begin{aligned}
          \mspan(V) & \supseteq \mspan \left( \begin{bmatrix} V_{1} & V_{2} \end{bmatrix}\right) &
          \text{and} &&
          \mspan(W) & \supseteq \mspan \left( W_{2} \right),
        \end{aligned}
      \end{equation*}
      and both of appropriate dimensions, then the symmetric
      transfer functions of $\hG$ interpolate those of $G$ in the following
      way:
      \begin{equation*}
        \begin{aligned}
          \symTF{1}(\sigma) & = \hsymTF{1}(\sigma), \\
          \symTF{1}(3 \sigma) & = \hsymTF{1}(3 \sigma), \\
          \symTF{2}(\sigma, \sigma) & = \hsymTF{2}(\sigma, \sigma), \\
          \symTF{3}(\sigma, \sigma, \sigma) &
            = \hsymTF{3}(\sigma, \sigma, \sigma).
        \end{aligned}
      \end{equation*}
  \end{enumerate}
\end{theorem}
\begin{proof}
  Part~(a) is an extension of \Cref{prp:symV} to the third symmetric subsystem
  transfer function with the special selection of interpolation points and
  follows immediately from~\cite[Thm.~6.2]{Wer21}.
  Part~(b) follows directly from \Cref{prp:symVW} such that only Part~(c) is
  left to be proven.
  The first three interpolation conditions follow from previous results,
  therefore we concentrate on the third symmetric subsystem transfer function.
  Inserting the selection of interpolation points into  \Cref{def:symTF}
  yields
  \begin{equation*}
    \begin{aligned}
      \hsymTF{3}(\sigma, \sigma, \sigma) & =
        \frac{1}{2} \hcC(3 \sigma) \hcK(3 \sigma)^{-1}
        \big( \hcH(2\sigma, \sigma) (\hV_{2} \otimes \hV_{1})\\
      & \quad{}+{} 
        \hcH(\sigma, 2\sigma) (\hV_{1} \otimes \hV_{2}) +
        \hcN(2\sigma) (I_{m} \otimes \hV_{2})\big),
    \end{aligned}
  \end{equation*}
  where
  \begin{equation*}
    \begin{aligned}
      \hV_{1} & = \hcK(\sigma)^{-1} \hcB(\sigma) \quad\text{and}\\
      \hV_{2} & = \hcK(2\sigma)^{-1} \big( \hcH(\sigma, \sigma)
        (\hV_{1} \otimes \hV_{1}) + \hcN(\sigma) (I_{m} \otimes \hV_{1}) \big).
    \end{aligned}
  \end{equation*}
  Using the projector $P_{\mathrm{V}} = V (W^{\herm} \cK(\sigma) V)^{-1}
  W^{\herm} \cK(\sigma)$ onto the space spanned by the columns of $V$, it
  follows that
  \begin{equation*}
    \begin{aligned}
      V \hV_{1} & = V_{1} & \text{and} &&
      V \hV_{2} & = V_{2}
    \end{aligned}
  \end{equation*}
  hold due to the construction of the space spanned by the columns of $V$ via
  the columns of $V_{1}$ and $V_{2}$.
  Using the projector $P_{\mathrm{W}} = W (W^{\herm} \cK(\sigma) V)^{-\herm}
  V^{\herm} \cK(\sigma)^{\herm}$ onto the space spanned by the columns of
  $W$, it also holds that
  \begin{equation*}
    \hcC(\sigma) \hcK(\sigma)^{-1} W^{\herm} = \cC(\sigma) \cK(\sigma)^{-1},
  \end{equation*}
  such that the final result of the theorem holds via
  \begin{equation*}
    \begin{aligned}
      \hsymTF{3}(\sigma, \sigma, \sigma) & =
        \frac{1}{2} \hcC(3 \sigma) \hcK(3 \sigma)^{-1} W^{\herm} 
        \big( \cH(2\sigma, \sigma) (V \hV_{2} \otimes V \hV_{1})\\
      & \quad{}+{} 
        \cH(\sigma, 2\sigma) (V\hV_{1} \otimes V\hV_{2}) +
        \cN(2\sigma) (I_{m} \otimes V\hV_{2})\big)\\
      & = \frac{1}{2} \cC(3 \sigma) \cK(3 \sigma)^{-1}
        \big( \cH(2\sigma, \sigma) (V_{2} \otimes V_{1})\\
      & \quad{}+{} 
        \cH(\sigma, 2\sigma) (V_{1} \otimes V_{2}) +
        \cN(2\sigma) (I_{m} \otimes V_{2})\big) \\
      & = \symTF{3}(\sigma, \sigma, \sigma).
    \end{aligned}
  \end{equation*}
\end{proof}


\subsection{Interpolating generalized transfer functions}

In this section, we investigate conditions on the projection spaces
for the interpolation of the structured generalized transfer functions from
\Cref{def:genTF}.
The following proposition can be seen as an analog to \Cref{prp:symV} for the
generalized case.

\begin{proposition}[{\cite[Thm.~6.13]{Wer21}}]%
  \label{prp:genV}
  Let $G$ be a quadratic-bilinear system, described by its generalized
  transfer functions $\genTF{k}{(.)}$ from \Cref{def:genTF}, and $\hG$ the
  reduced-order quadratic-bilinear system constructed by~\cref{eqn:projection},
  with its reduced-order generalized transfer functions $\hgenTF{k}{(.)}$.
  Also, let $\sigma_{1}, \sigma_{2}, \sigma_{3} \in \C$ be interpolation points
  such that the matrix-valued functions $\cC, \cK, \cB, \cN, \cH$ and
  $\cK(.)^{-1}$ are defined at these points.
  Compute
  \begin{equation*}
    \begin{aligned}
      V_{1, 1} & = \cK(\sigma_{1})^{-1} \cB(\sigma_{1}), \\
      V_{1, 2} & = \cK(\sigma_{2})^{-1} \cB(\sigma_{2}), \\
      V_{2} & = \cK(\sigma_{2})^{-1} \cN(\sigma_{1})
        (I_{m} \otimes V_{1, 1}),\\
      V_{3, 1} & = \cK(\sigma_{3})^{-1} \cN(\sigma_{2}) (I_{m} \otimes V_{2}),\\
      V_{3, 2} & = \cK(\sigma_{3})^{-1} \cH(\sigma_{2}, \sigma_{1})
        (V_{1, 2} \otimes V_{1, 1}),
    \end{aligned}
  \end{equation*}
  and construct the basis matrix $V$ such that
    \begin{equation*}
    \begin{aligned}
      \mspan(V) & \supseteq \mspan \left(
        \begin{bmatrix} V_{1, 1} & V_{1, 2} & V_{2} &
        V_{3, 1} & V_{3, 2} \end{bmatrix} \right).
    \end{aligned}
  \end{equation*}
  Let $W$ be an arbitrary full-rank matrix of appropriate dimensions.
  Then, the generalized transfer functions of $\hG$ interpolate those of $G$
  in the following way:
  \begin{equation*}
    \begin{aligned}
      \genTF{1}{(B)}(\sigma_{1}) & = \hgenTF{1}{(B)}(\sigma_{1}), \\
      \genTF{1}{(B)}(\sigma_{2}) & = \hgenTF{1}{(B)}(\sigma_{2}), \\
      \genTF{2}{(N, (B))}(\sigma_{1}, \sigma_{2}) & =
        \hgenTF{2}{(N, (B))}(\sigma_{1}, \sigma_{2}), \\
      \genTF{3}{(N, (N, (B)))}(\sigma_{1}, \sigma_{2}, \sigma_{3}) & =
        \hgenTF{3}{(N, (N, (B)))}(\sigma_{1}, \sigma_{2}, \sigma_{3}), \\
      \genTF{3}{(H, (B), (B))}(\sigma_{1}, \sigma_{2}, \sigma_{3}) & =
        \hgenTF{3}{(H, (B), (B))}(\sigma_{1}, \sigma_{2}, \sigma_{3}).
    \end{aligned}
  \end{equation*}
\end{proposition}

\begin{remark}
  As for the symmetric subsystem transfer functions, it is possible to reduce
  the dimensions of the constructed projection space by choosing suitable
  interpolation points and, also as in the symmetric subsystem transfer function
  case, this can be achieved by choosing $\sigma_{1} = \sigma_{2} = \sigma_{3}$.
  However, only marginal savings in terms of basis contributions and
  computational costs can be achieved by this since in
  \Cref{prp:genV}, only the matrix $V_{1, 2}$ can be omitted.
\end{remark}

Note that in \Cref{prp:genV}, the basis contribution from $V_{3, 1}$ results
in the interpolation of the third generalized transfer function with two
bilinear terms $\genTF{3}{(N, (N, (B)))}$.
It may happen that no interpolation conditions are imposed for this transfer
function such that the subspace dimensions can be reduced by omitting
$V_{3, 1}$.
As in the case of symmetric subsystem transfer functions, the second basis
matrix $W$ can be used to reduce the minimal dimension of the constructed
subspaces and to simplify the construction of the subspaces.
These results are given in the following corollary.

\begin{corollary}%
  \label{cor:genVW}
  Let $G$ be a quadratic-bilinear system, described by its generalized
  transfer functions $\genTF{k}{(.)}$ from \Cref{def:genTF}, and $\hG$ the
  reduced-order quadratic-bilinear system constructed by~\cref{eqn:projection},
  with its reduced-order generalized transfer functions $\hgenTF{k}{(.)}$.
  Also, let $\sigma_{1}, \sigma_{2} \in \C$ be interpolation points
  such that the matrix-valued functions $\cC, \cK, \cB, \cN, \cH$ and
  $\cK(.)^{-1}$ are defined at these points.
  Let the basis matrices $V$ and $W$ be constructed by
  \begin{equation*}
    \begin{aligned}
      \mspan(V) & \supseteq
        \mspan \left( \cK(\sigma_{1})^{-1} \cB(\sigma_{1}) \right), \\
      \mspan(W) & \supseteq
        \mspan\left(\cK(\sigma_{2})^{-\herm} \cC(\sigma_{2})^{\herm}\right),
    \end{aligned}
  \end{equation*}
  and are of the same dimension.
  Then, the generalized transfer functions of $\hG$ interpolate those of $G$
  in the following way:
  \begin{equation*}
    \begin{aligned}
      \genTF{1}{(B)}(\sigma_{1}) & = \hgenTF{1}{(B)}(\sigma_{1}),\\
      \genTF{1}{(B)}(\sigma_{2}) & = \hgenTF{1}{(B)}(\sigma_{2}),\\
      \genTF{2}{(N, (B))}(\sigma_{1}, \sigma_{2}) & =
        \hgenTF{2}{(N, (B))}(\sigma_{1}, \sigma_{2}), \\
      \genTF{3}{(H, (B), (B))}(\sigma_{1}, \sigma_{1}, \sigma_{2}) & =
        \hgenTF{3}{(H, (B), (B))}(\sigma_{1}, \sigma_{1}, \sigma_{2}).
    \end{aligned}
  \end{equation*}
\end{corollary}
\begin{proof}
  The result follows directly from~\cite[Lem.~6.15]{Wer21} by restriction to
  two interpolation points.
\end{proof}

Similar to \Cref{prp:symVW}, the result in \Cref{cor:genVW} states that the
interpolation of higher level transfer functions is possible in an implicit
way without evaluating any of the nonlinear terms.
\Cref{cor:genVW} shows the version of the implicit interpolation result with
the smallest achievable minimal subspace dimensions for the interpolation of the
third level generalized transfer function with quadratic term.
The choice of identical interpolation points will not further reduce the
dimensions of the projection spaces, but allows to replace the interpolation
of the first level generalized transfer function in two points by matching
the transfer function value and its derivative in one point;
see~\cite{BeaG09} and~\cite[Lem.~6.15]{Wer21}.


\section{Numerical experiments}%
\label{sec:examples}

Now, we employ the interpolation results from above for constructing
structured reduced-order quadratic-bilinear systems in two numerical examples.
The experiments were run on compute nodes of the \texttt{Greene}
high-performance computing cluster of the New York University using 16
processing cores of the Intel Xeon Platinum 8268 24C 205W CPU at 2.90\,GHz
and 32\,GB main memory.
We used \matlab{} 9.9.0.1467703 (R2020b) running on Red Hat Enterprise Linux
release 8.4 (Ootpa).
The source code, data and results of the numerical experiments are
open source/open access and available at~\cite{supWer23b}.


\subsection{Experimental setup}%
\label{subsec:setup}

In both numerical examples, we compute reduced-order models via
structure-preserving interpolation of the symmetric subsystem transfer functions
and the generalized transfer functions, denoted by \texttt{SymInt}
and \texttt{GenInt}, respectively.
We compute models using either (i) only the construction of the basis matrix $V$
and a one-sided projection by setting $W = V$, which we abbreviate further on
by \texttt{V}, or (ii) by also constructing the left basis matrix $W$ for a
two-sided projection following the results in \Cref{thm:symVW,cor:genVW}, which
we abbreviate by \texttt{VW}.
For simplicity, the interpolation points are chosen logarithmically equidistant
on the imaginary axis in all cases.
If we compute only the basis matrices for interpolation without additional
information, this is denoted by \texttt{equi}.
On the other hand, if we oversampled the frequency range of interest and
compressed the resulting basis to a prescribed dimension, e.g., using pivoted
QR, this is denoted by \texttt{avg}; cf.~\cite[Rem.~3.3]{Wer21}.
In all cases, we focus on the interpolation of either (i) the first two
symmetric subsystem transfer functions or (ii) the first two levels of the
generalized transfer functions and the third level generalized transfer
function containing the quadratic term.
The following overview summarizes the considered interpolation methods:
\begin{description}
  \item[\symintVequi{}] is the interpolation of symmetric subsystem transfer
    functions via one-sided projection by constructing the basis matrix $V$.
  \item[\symintVWequi{}] is the interpolation of symmetric subsystem transfer
    functions via two-sided projection.
    Additional interpolation points are selected for the construction of $W$ to
    match the dimension of $V$.
  \item[\symintVavg{}] is the approximation of an interpolation basis for
    symmetric subsystem transfer functions using only samples for the
    construction of $V$ and one-sided projection.
  \item[\symintVWavg{}] is the approximation of left and right interpolation
    bases for symmetric subsystem transfer functions using samples for the
    construction of $V$ and $W$ and two-sided projection.
    Additional interpolation points are selected for the construction of $W$ to
    match the computational work to the construction of $V$.
  \item[\genintVequi{}] is the interpolation of generalized transfer
    functions via one-sided projection by constructing the basis matrix $V$.
    Samples from the second and third level transfer functions are taken
    alternating.
  \item[\genintVWequi{}] is the interpolation of generalized transfer
    functions via two-sided projection.
    For the construction of $V$, samples from the second and third level
    transfer functions are taken alternating.
    Additional interpolation points are selected for the construction of $W$ to
    match the dimension of $V$.
  \item[\genintVavg{}] is the approximation of an interpolation basis for
    generalized transfer functions using only samples for the
    construction of $V$ and one-sided projection.
    At all interpolation points, second and third level samples are taken.
  \item[\symintVWavg{}] is the approximation of left and right interpolation
    bases for generalized transfer functions using samples for the
    construction of $V$ and $W$ and two-sided projection.
    For the construction of $V$, second and third level samples are taken at
    all interpolation points.
    Additional interpolation points are selected for the construction of $W$ to
    equalize the computational work to the construction of $V$.
\end{description}

As an additional comparison, we have computed reduced-order models via proper
orthogonal decomposition (POD).
All POD models have been trained via simulations of the unit step response
to remove the correlation of the training and test input signals.
For a fair comparison, the trajectory lengths used for POD are chosen with
respect to comparable amounts of computational work to the interpolation
methods.
We consider therefore:
\begin{description}
  \item[\pod{},] which has a computational workload similar to \symintVequi{}
    and \genintVequi{}, and
  \item[\podavg{},] which uses similar to \symintVavg{} and \genintVavg{}
    an oversampling and computes the orthogonal basis via truncated singular
    value decomposition.
\end{description}

For the comparison of the reduced-order models in time domain, we simulate the
models over finite time intervals using input signals taken from a
Gaussian process $\GP(\mu, \mathsf{K})$, with constant mean $\mu \in \R$
and the squared exponential kernel
\begin{equation*}
  \mathsf{K}(x, y) = \exp\left( -\frac{\lvert x - y \rvert^{2}}
    {2 \varsigma^{2}} \right),
\end{equation*}
where $\varsigma \geq 0$ is a smoothing parameter.
The parameters $\mu$ and $\varsigma$ are chosen independently for the two
examples and are given below.
For visualization, we compute and plot the maximum pointwise relative errors
\begin{equation*}
  \relerr(t) := \max_{j} \left\lvert \frac{y_{j}(t) - \hy_{j}(t)}{y_{j}(t)}
    \right\rvert.
\end{equation*}
Also, we compute discretized approximations of the relative $L_{2}$ and
$L_{\infty}$ errors via
\begin{equation*}
  \begin{aligned}
    \relerr_{L_{2}} & := \frac{\lVert \mvec(y_{\mathrm{h}} -
      \hy_{\mathrm{h}}) \rVert_{2}}{\lVert \mvec(y_{\mathrm{h}}) \rVert_{2}}
      &  \text{and} &&
    \relerr_{L_{\infty}} & := \frac{\lVert \mvec(y_{\mathrm{h}} -
      \hy_{\mathrm{h}}) \rVert_{\infty}}{\lVert \mvec(y_{\mathrm{h}})
      \rVert_{\infty}},
  \end{aligned}
\end{equation*}
where $y_{\mathrm{h}}, \hy_{\mathrm{h}} \in \R^{p \times n_{h}}$ are the
discretized output signals of the original and reduced-order model,
respectively, in the time interval $[0, t_{\mathrm{f}}]$, and $\mvec(.)$
is the vectorization operator.

In frequency domain, we consider the pointwise relative spectral norm errors
defined as
\begin{equation*}
  \begin{aligned}
    \relerr(\omega) & := \frac{\lVert \symTF{1}(\i \omega) -
        \hsymTF{1}(\i \omega) \rVert_{2}}
        {\lVert \symTF{1}(\i \omega)\rVert_{2}} & \text{and} \\
    \relerr(\omega_{1}, \omega_{2}) & := \frac{\lVert
        \symTF{2}(\i \omega_{1}, \i \omega_{2}) -
        \hsymTF{2}(\i \omega_{1}, \i \omega_{2}) \rVert_{2}}
        {\lVert \symTF{2}(\i \omega_{1}, \i \omega_{2}) \rVert_{2}},
  \end{aligned}
\end{equation*}
over the limited frequency intervals $\omega, \omega_{1}, \omega_{2} \in
[\omega_{\min}, \omega_{\max}]$.
Additionally, we compute approximations to the relative $\Linf$-norm errors
for the first and second symmetric subsystem transfer functions via
\begin{equation*}
  \begin{aligned}
    \relerr_{\Linf}^{(1)} & :=
      \frac{\max\limits_{\omega} \lVert \symTF{1}(\i \omega) -
      \hsymTF{1}(\i \omega) \rVert_{2}}
      {\max\limits_{\omega} \lVert \symTF{1}(\i \omega)\rVert_{2}} &
    \text{and} \\
    \relerr_{\Linf}^{(2)} & := \frac{\max\limits_{\omega_{1},\, \omega_{2}}
      \lVert \symTF{2}(\i \omega_{1}, \i \omega_{2}) -
      \hsymTF{2}(\i \omega_{1}, \i \omega_{2}) \rVert_{2}}
      {\max\limits_{\omega_{1},\, \omega_{2}}
      \lVert \symTF{2}(\i \omega_{1}, \i \omega_{2}) \rVert_{2}},
  \end{aligned}
\end{equation*}
using $500$ logarithmically equidistant sampling points in the frequency
interval of interest~$[\omega_{\min}, \omega_{\max}]$.

For further details on the experimental setup, we refer the reader to the
accompanying code package~\cite{supWer23b}.


\subsection{Quadratic time-delayed reaction-diffusion model}%
\label{subsec:timedelay}

\begin{table}[t]
  \centering
  \caption{Errors computed as shown in \Cref{subsec:setup} for the time-delay
    example:
    \pod{} computes the worst performing reduced-order model, while
    \podavg{} is only as good as the worst interpolation-based models.
    Here, \symintVWavg{} performs best w.r.t.\ three out of the four error
    measures, while according to $\relerr_{\Linf}^{(1)}$, the best
    reduced-order model is computed by \genintVWavg{}.}
  \label{tab:timedelay}
  \vspace{.5\baselineskip}
  
  \begin{tabular}{lrrrr}
    \hline\noalign{\smallskip}
    &
      \multicolumn{1}{c}{$\boldsymbol{\relerr_{L_{2}}}$} &
      \multicolumn{1}{c}{$\boldsymbol{\relerr_{L_{\infty}}}$} & 
      \multicolumn{1}{c}{$\boldsymbol{\relerr_{\Linf}^{(1)}}$} &
      \multicolumn{1}{c}{$\boldsymbol{\relerr_{\Linf}^{(2)}}$} \\
    \noalign{\smallskip}\hline\noalign{\medskip}
    \symintVequi{} &
      $8.1604\texttt{e-}07$ &
      $1.8475\texttt{e-}06$ &
      $2.5851\texttt{e-}06$ &
      $3.4393\texttt{e-}06$ \\
    \symintVavg{} &
      $3.1414\texttt{e-}08$ &
      $4.9411\texttt{e-}08$ &
      $3.4598\texttt{e-}08$ &
      $1.0957\texttt{e-}06$ \\
    \symintVWequi{} &
      $1.0181\texttt{e-}05$ &
      $4.9409\texttt{e-}05$ &
      $1.0705\texttt{e-}07$ &
      $4.3354\texttt{e-}06$ \\
    \symintVWavg{} &
      $6.1325\texttt{e-}09$ &
      $2.4509\texttt{e-}08$ &
      $4.9776\texttt{e-}10$ &
      $2.5239\texttt{e-}09$ \\
    \noalign{\medskip}\hline\noalign{\medskip}
    \genintVequi{} &
      $3.2373\texttt{e-}06$ &
      $4.3868\texttt{e-}06$ &
      $1.1713\texttt{e-}05$ &
      $5.3102\texttt{e-}06$ \\
    \genintVavg{} &
      $3.9579\texttt{e-}08$ &
      $6.7805\texttt{e-}08$ &
      $3.1402\texttt{e-}08$ &
      $1.0562\texttt{e-}06$ \\
    \genintVWequi{} &
      $1.1332\texttt{e-}05$ &
      $2.9810\texttt{e-}05$ &
      $7.0970\texttt{e-}07$ &
      $2.1736\texttt{e-}06$ \\
    \genintVWavg{} &
      $1.0280\texttt{e-}08$ &
      $4.3505\texttt{e-}08$ &
      $1.1418\texttt{e-}10$ &
      $4.1712\texttt{e-}09$ \\
    \noalign{\medskip}\hline\noalign{\medskip}
    \pod{} &
      $5.2902\texttt{e-}04$ &
      $1.0422\texttt{e-}03$ &
      $6.2910\texttt{e-}04$ &
      $2.8691\texttt{e-}04$ \\
    \podavg{} &
      $2.1228\texttt{e-}05$ &
      $5.1048\texttt{e-}05$ &
      $9.8140\texttt{e-}06$ &
      $3.7744\texttt{e-}05$ \\
    \noalign{\medskip}\hline\noalign{\smallskip}
  \end{tabular}
\end{table}

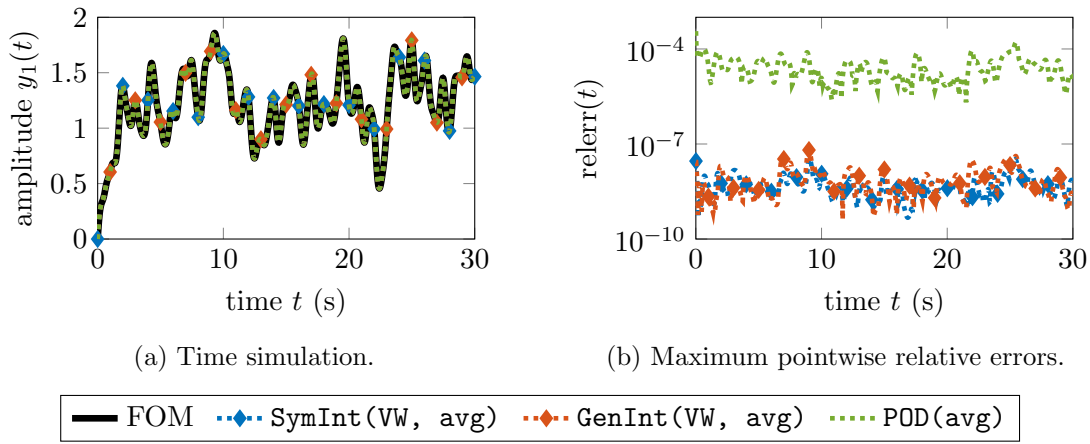
\begin{figure}[t]
  \centering
  \begin{subfigure}[b]{.49\textwidth}
    \centering
  \tikzexternalenable%
  \tikzsetnextfilename{timedelay_time_sim_best}%
  \begin{tikzpicture}
  \pgfplotstableread{graphics/data/timedelay_time_sim_best.dat}\tableSIM
  
  \begin{axis}[%
    width  = .675\textwidth,
    height = .4\textwidth,
    scale only axis,
    xmin = 0,
    xmax = 30,
    ymin = 0,
    ymax = 2,
    xminorticks = false,
    yminorticks = false,
    xlabel = {time $t$ (s)},
    ylabel = {amplitude $y_{1}(t)$},
    ylabel style   = {yshift = -.3em},
    scaled x ticks = false,
    x tick label style = {/pgf/number format/fixed},
    clip mode          = individual]
    
    \addplot[fom] table[x index = 0, y index = 1] {\tableSIM};
    \addplot[symint, VW, avg, mark repeat = 200]
      table[x index = 0, y index = 3] {\tableSIM};
    \addplot[genint, VW, avg, mark repeat = 200, mark phase = 100]
      table[x index = 0, y index = 5] {\tableSIM};
    \addplot[pod, avg] table[x index = 0, y index = 7] {\tableSIM};
  \end{axis}
\end{tikzpicture}%
  \tikzexternaldisable%

    \caption{Time simulation.}
    \label{fig:timedelay_time_sim_best}
  \end{subfigure}%
  \hfill%
  \begin{subfigure}[b]{.49\textwidth}
    \centering
  \tikzexternalenable%
  \tikzsetnextfilename{timedelay_time_relerr_best}%
  \begin{tikzpicture}
  \pgfplotstableread{graphics/data/timedelay_time_relerr_best.dat}\tableERR
  
  \begin{semilogyaxis}[%
    width  = .675\textwidth,
    height = .4\textwidth,
    scale only axis,
    xmin = 0,
    xmax = 30,
    ymin = 1e-10,
    ymax = 1e-3,
    xminorticks = false,
    yminorticks = false,
    xlabel = {time $t$ (s)},
    ylabel = {$\relerr(t)$},
    ylabel style   = {yshift = -.3em},
    scaled x ticks = false,
    x tick label style = {/pgf/number format/fixed},
    clip mode          = individual]
    
    \addplot[symint, VW, avg, mark repeat = 200]
      table[x index = 0, y index = 1] {\tableERR};
    \addplot[genint, VW, avg, mark repeat = 200, mark phase = 100]
      table[x index = 0, y index = 2] {\tableERR};
    \addplot[pod, avg] table[x index = 0, y index = 3] {\tableERR};
  \end{semilogyaxis}
\end{tikzpicture}%
  \tikzexternaldisable%

    \caption{Maximum pointwise relative errors.}
    \label{fig:timedelay_time_relerr_best}
  \end{subfigure}
  
  \vspace{.5\baselineskip}
  \tikzexternalenable%
  \tikzsetnextfilename{timedelay_time_legend}%
  \begin{tikzpicture}
  \begin{axis}[%
    hide axis,
    width  = 1mm,
    height = 1mm,
    scale only axis,
    xmin = 0,
    xmax = 1,
    ymin = 0,
    ymax = 1,
    legend columns = 4, 
    legend style   = {
      at     = {(0,0)},
      anchor = center,
      /tikz/every even column/.append style = {column sep = 0.2cm}},
    legend cell align  = {left},
    clip mode          = individual]
    
    \addlegendimage{fom}
    \addlegendentry{FOM}
    
    \addlegendimage{symint, VW, avg}
    \addlegendentry{\symintVWavg}
    
    \addlegendimage{genint, VW, avg}
    \addlegendentry{\genintVWavg}
    
    \addlegendimage{pod, avg}
    \addlegendentry{\podavg}
  \end{axis}
\end{tikzpicture}%
  \tikzexternaldisable%

  \caption{Time simulation of the time-delay example: The best reduced-order
    models from each generating approach are shown.
    All reduced-order models can recover the system behavior for the given
    input signal, but the interpolation-based reduced-order models
    perform around four orders of magnitude better in terms of accuracy than the
    model generated by \podavg{}.}
  \label{fig:timedelay_time}
\end{figure}

\begin{figure}[t]
  \centering
  \begin{subfigure}[b]{.49\textwidth}
    \centering
  \tikzexternalenable%
  \tikzsetnextfilename{timedelay_g1_tf_best}%
  \begin{tikzpicture}
  \pgfplotstableread{graphics/data/timedelay_g1_tf_best.dat}\tableTF
  
  \begin{loglogaxis}[%
    width  = .675\textwidth,
    height = .4\textwidth,
    scale only axis,
    xmin = 1e-3,
    xmax = 1e+3,
    ymin = 1e-3,
    ymax = 3e+0,
    xminorticks = false,
    yminorticks = false,
    xlabel = {frequency $\omega$ (rad/s)},
    ylabel = {magnitude},
    ylabel style   = {yshift = -.3em},
    scaled x ticks = false,
    x tick label style = {/pgf/number format/fixed},
    clip mode          = individual]
    
    \addplot[fom] table[x index = 0, y index = 1] {\tableTF};
    \addplot[symint, VW, avg, mark repeat = 50]
      table[x index = 0, y index = 2] {\tableTF};
    \addplot[genint, VW, avg, mark repeat = 50, mark phase = 20]
      table[x index = 0, y index = 3] {\tableTF};
    \addplot[pod, avg] table[x index = 0, y index = 4] {\tableTF};
  \end{loglogaxis}
\end{tikzpicture}%
  \tikzexternaldisable%

    \caption{Linear transfer functions.}
    \label{fig:timedelay_g1_tf_best}
  \end{subfigure}%
  \hfill%
  \begin{subfigure}[b]{.49\textwidth}
    \centering
  \tikzexternalenable%
  \tikzsetnextfilename{timedelay_g1_relerr_best}%
  \begin{tikzpicture}
  \pgfplotstableread{graphics/data/timedelay_g1_relerr_best.dat}\tableERR
  
  \begin{loglogaxis}[%
    width  = .675\textwidth,
    height = .4\textwidth,
    scale only axis,
    xmin = 1e-3,
    xmax = 1e+3,
    ymin = 1e-14,
    ymax = 1e-2,
    xminorticks = false,
    yminorticks = false,
    xlabel = {frequency $\omega$ (rad/s)},
    ylabel = {$\relerr(\omega)$},
    ylabel style   = {yshift = -.3em},
    scaled x ticks = false,
    x tick label style = {/pgf/number format/fixed},
    clip mode          = individual]
    
    \addplot[symint, VW, avg, mark repeat = 50]
      table[x index = 0, y index = 1] {\tableERR};
    \addplot[genint, VW, avg, mark repeat = 50, mark phase = 20]
      table[x index = 0, y index = 2] {\tableERR};
    \addplot[pod, avg] table[x index = 0, y index = 3] {\tableERR};
  \end{loglogaxis}
\end{tikzpicture}%
  \tikzexternaldisable%

    \caption{Pointwise relative errors.}
    \label{fig:timedelay_g1_relerr_best}
  \end{subfigure}
  
  \vspace{.5\baselineskip}
  \tikzexternalenable%
  \tikzsetnextfilename{timedelay_g1_legend}%
  \begin{tikzpicture}
  \begin{axis}[%
    hide axis,
    width  = 1mm,
    height = 1mm,
    scale only axis,
    xmin = 0,
    xmax = 1,
    ymin = 0,
    ymax = 1,
    legend columns = 4, 
    legend style   = {
      at     = {(0,0)},
      anchor = center,
      /tikz/every even column/.append style = {column sep = 0.2cm}},
    legend cell align  = {left},
    clip mode          = individual]
    
    \addlegendimage{fom}
    \addlegendentry{FOM}
    
    \addlegendimage{symint, VW, avg}
    \addlegendentry{\symintVWavg}
    
    \addlegendimage{genint, VW, avg}
    \addlegendentry{\genintVWavg}
    
    \addlegendimage{pod, avg}
    \addlegendentry{\podavg}
  \end{axis}
\end{tikzpicture}%
  \tikzexternaldisable%

  \caption{Sigma plots showing $\lVert G(\i \omega)) \rVert_{2}$ of the first
    symmetric subsystem transfer function of the time-delay example:
    The best reduced-order models from each generating approach are shown.
    All reduced-order models can recover the system behavior for the given
    input signal, but the interpolation-based reduced-order models
    perform around six orders of magnitude better in terms of accuracy than the
    model generated by \podavg{}.}
  \label{fig:timedelay_g1}
\end{figure}
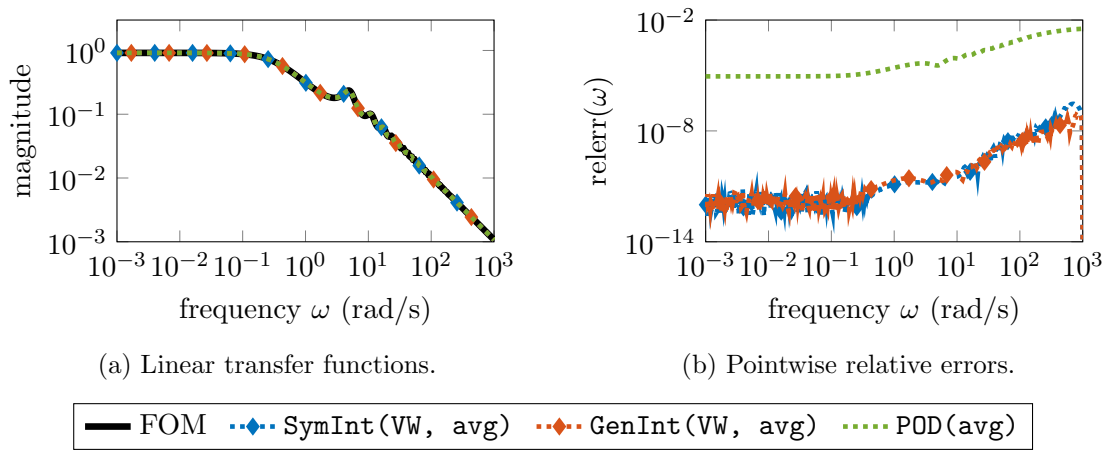

\begin{figure}[t]
  \centering
  \begin{subfigure}[b]{.49\textwidth}
    \centering
  \tikzexternalenable%
  \tikzsetnextfilename{timedelay_g2_symvwavg}%
  \begin{tikzpicture}
  \begin{loglogaxis}[
    view   = {0}{90},
    width  = .675\textwidth,
    height = .4\textwidth,
    scale only axis,
    axis on top,
    xmin   = 1e-3,
    xmax   = 1e+3,
    ymin   = 1e-3,
    ymax   = 1e+3,
    xtick  = {1e-3, 1e-1, 1e+1, 1e+3},
    ytick  = {1e-3, 1e-1, 1e+1, 1e+3},
    xminorticks = false,
    yminorticks = false,
    xlabel = {frequency $\omega_{1}$ (rad/s)},
    ylabel = {frequency $\omega_{2}$ (rad/s)},
    ylabel style = {yshift = -.3em},
    scaled x ticks = false,
    x tick label style = {/pgf/number format/fixed}]
        
      \addplot graphics[xmin = 1e-3, xmax = 1e+3, ymin = 1e-3, ymax = 1e+3]
        {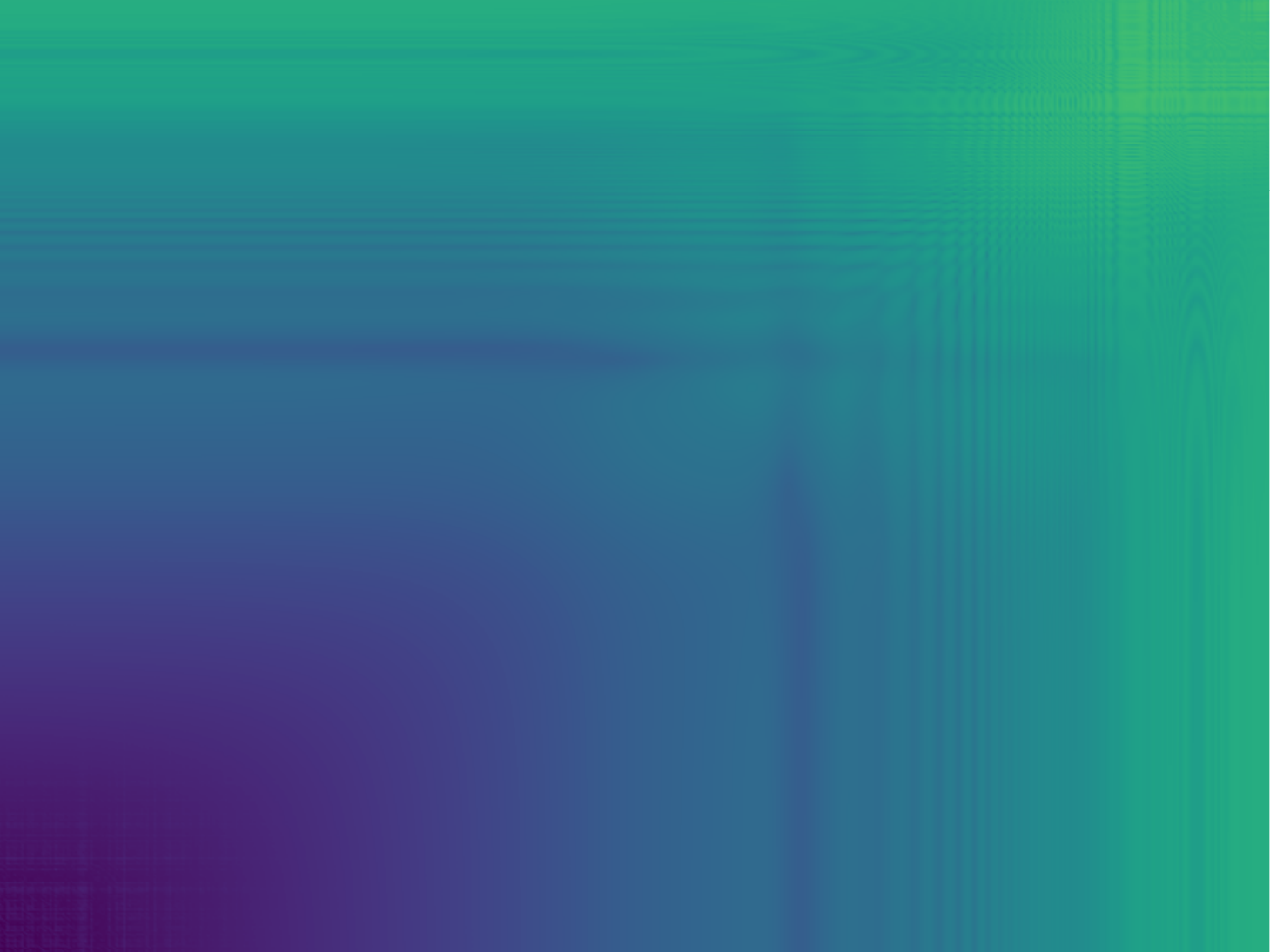};
            
  \end{loglogaxis}
\end{tikzpicture}%
  \tikzexternaldisable%

    \caption{\symintVWavg{}.}
    \label{fig:timedelay_g2_symvwavg}
  \end{subfigure}%
  \hfill%
  \begin{subfigure}[b]{.49\textwidth}
    \centering
  \tikzexternalenable%
  \tikzsetnextfilename{timedelay_g2_genvwavg}%
  \begin{tikzpicture}
  \begin{loglogaxis}[
    view   = {0}{90},
    width  = .675\textwidth,
    height = .4\textwidth,
    scale only axis,
    axis on top,
    xmin   = 1e-3,
    xmax   = 1e+3,
    ymin   = 1e-3,
    ymax   = 1e+3,
    xtick  = {1e-3, 1e-1, 1e+1, 1e+3},
    ytick  = {1e-3, 1e-1, 1e+1, 1e+3},
    xminorticks = false,
    yminorticks = false,
    xlabel = {frequency $\omega_{1}$ (rad/s)},
    ylabel = {frequency $\omega_{2}$ (rad/s)},
    ylabel style = {yshift = -.3em},
    scaled x ticks = false,
    x tick label style = {/pgf/number format/fixed}]
        
      \addplot graphics[xmin = 1e-3, xmax = 1e+3, ymin = 1e-3, ymax = 1e+3]
        {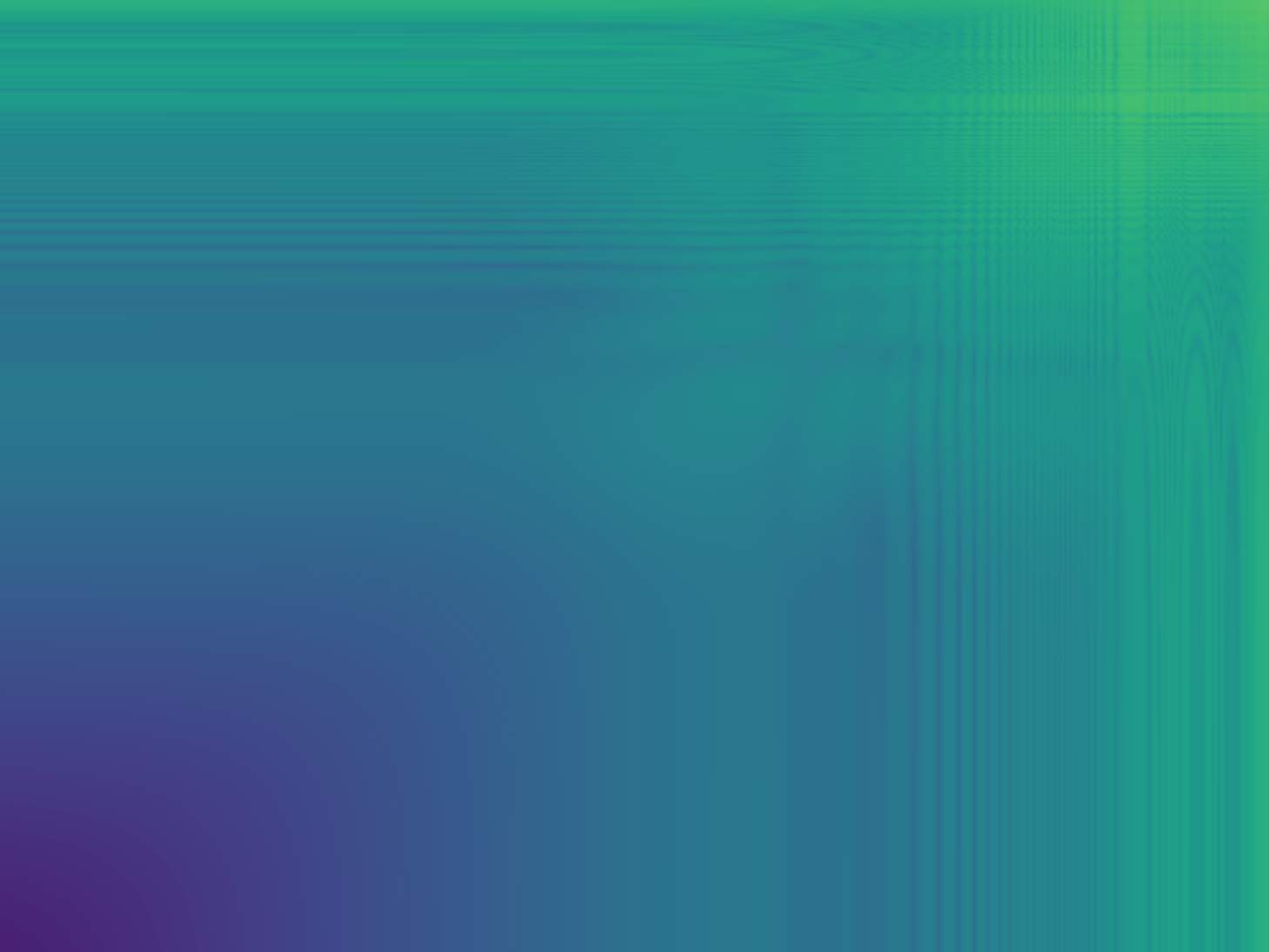};
            
  \end{loglogaxis}
\end{tikzpicture}%
  \tikzexternaldisable%

    \caption{\genintVWavg{}.}
    \label{fig:timedelay_g2_genvwavg}
  \end{subfigure}
  
  \vspace{.5\baselineskip}
  \begin{subfigure}{.49\textwidth}
    \centering
  \tikzexternalenable%
  \tikzsetnextfilename{timedelay_g2_podavg}%
  \begin{tikzpicture}
  \begin{loglogaxis}[
    view   = {0}{90},
    width  = .675\textwidth,
    height = .4\textwidth,
    scale only axis,
    axis on top,
    xmin   = 1e-3,
    xmax   = 1e+3,
    ymin   = 1e-3,
    ymax   = 1e+3,
    xtick  = {1e-3, 1e-1, 1e+1, 1e+3},
    ytick  = {1e-3, 1e-1, 1e+1, 1e+3},
    xminorticks = false,
    yminorticks = false,
    xlabel = {frequency $\omega_{1}$ (rad/s)},
    ylabel = {frequency $\omega_{2}$ (rad/s)},
    ylabel style = {yshift = -.3em},
    scaled x ticks = false,
    x tick label style = {/pgf/number format/fixed}]
        
      \addplot graphics[xmin = 1e-3, xmax = 1e+3, ymin = 1e-3, ymax = 1e+3]
        {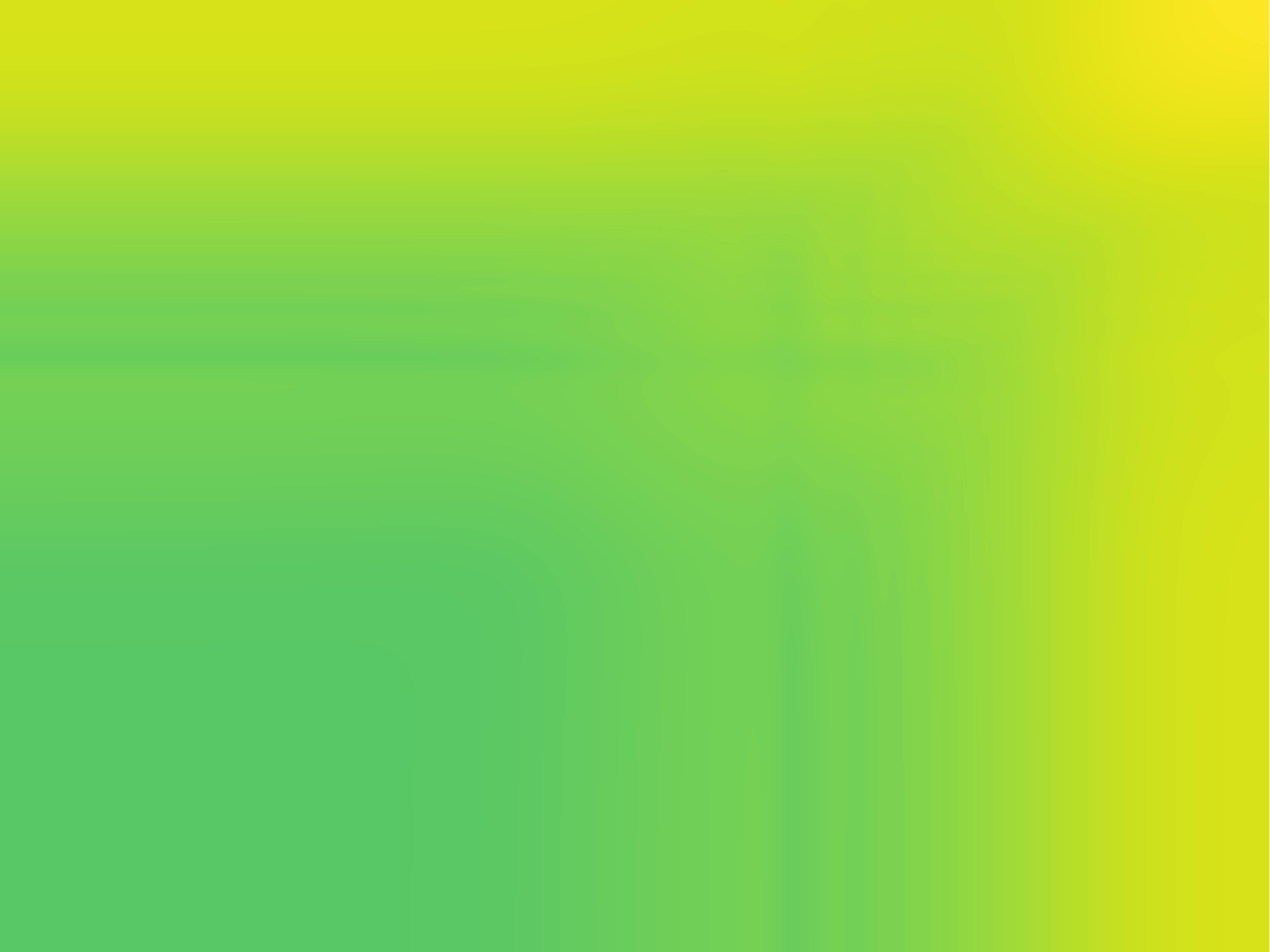};
            
  \end{loglogaxis}
\end{tikzpicture}%
  \tikzexternaldisable%

    \caption{\podavg{}.}
    \label{fig:timedelay_g2_podavg}
  \end{subfigure}%
  \hfill%
  \begin{subfigure}{.49\textwidth}
    \raggedleft
  \tikzexternalenable%
  \tikzsetnextfilename{timedelay_g2_legend}%
  \begin{tikzpicture}
  \node[draw = none, minimum width = 0cm, inner sep = 0cm](start){};
  \node(leg) at (start.north east) [anchor = north west]{\tikz
  \begin{axis}[%
    hide axis,
    scale only axis,
    width  = .8\textwidth,
    height = .1cm,
    point meta min = -11.5618,
    point meta max = -1.9278,
    colorbar,
    colorbar horizontal,
    colorbar style = {
      xticklabel = $10^{\pgfmathparse{\tick}
        \pgfmathprintnumber\pgfmathresult}$,
      at     = {(.5, 0)},
      anchor = north},
    scaled x ticks = false,
    x tick label style = {/pgf/number format/fixed}]
  \end{axis};};
  \node[draw = none, minimum width = 0cm, inner sep = 0cm](end)
    at (leg.north east) [anchor = north west]{};
\end{tikzpicture}%
  \tikzexternaldisable%

    \vspace{2\baselineskip}
  \end{subfigure}
  
  \caption{Second symmetric subsystem transfer function relative approximation
    errors $\relerr(\omega_{1}, \omega_{2})$ of the time-delay example:
    The best reduced-order models from each generating approach are shown.
    The errors of both interpolation-based reduced-order models are at least
    four orders of magnitude better than those of the \podavg{} model.}
  \label{fig:timedelay_g2}
\end{figure}
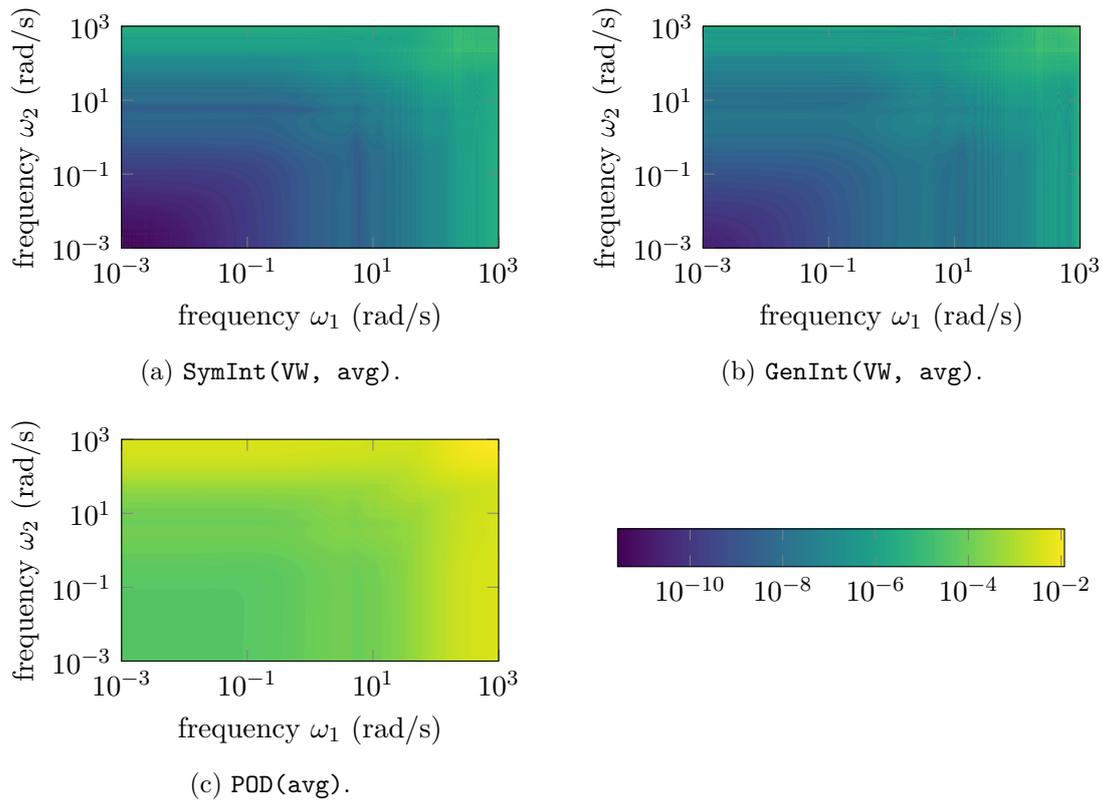

As first example, we consider the time-delayed heated rod with bilinear
feedback from~\cite{GosPBetal19,BenGW22}, to which we append a quadratic
reaction term to obtain
\begin{equation*}
  \begin{aligned}
    \partial_{\rm{t}} \nu(\zeta, t) & = \Delta \nu(\zeta, t)
      - 2 \sin(\zeta) \nu(\zeta, t) 
      + 2 \sin(\zeta) \nu(\zeta, t - \tau)\\
    & \quad{}-{}
      2 \sin(\zeta) \nu(\zeta, t)^{2}
      + \sum\limits_{j = 1}^{m} b_{j}(\zeta)(\nu(\zeta, t)  + 1) u_{j}(t),
  \end{aligned}
\end{equation*}
with $(t, \zeta) \in (0, t_{\mathrm{f}}) \times (0, \pi)$, boundary
conditions $\nu(t, 0) = \nu(t, \pi) = 0$ for all $t \in [0, t_{\mathrm{f}}]$,
and the constant time delay $\tau = 1$.
After spatial discretization with central finite differences, we obtain a
quadratic-bilinear time-delay system of the form
\begin{equation*}
  \begin{aligned}
    E \dot{x}(t) & = A x(t) + A_{\mathrm{d}} x(t - \tau) + H (x(t) \otimes x(t))
      + \sum\limits_{j = 1}^{m} N_{k} x(t) u_{j}(t) + B u(t),\\
    y(t) & = C x(t),
  \end{aligned}
\end{equation*}
with $E, A, A_{\mathrm{d}}, N_{j} \in \R^{n \times n}$, for $j = 1, \ldots, m$,
$H \in \R^{n \times n^{2}}$, $B \in \R^{n \times m}$ and
$C \in \R^{p \times n}$.
For our experiments, we have chosen $n = 2\,000$, $m = 2$ and $p = 2$, where
$u_{1}$ controls the temperature of the first third of the rod and $u_{2}$ the
rest, and $y_{1}$ observes the temperature of the first half of the rod and
$y_{2}$ of the second half.
The system has zero initial conditions $x(t) = 0$ for all $t \leq 0$.

The reduced-order models are computed as explained in \Cref{subsec:setup} with
a reduced order of $r = 24$.
In \Cref{tab:timedelay}, we can see that all reduced-order models perform well
for this example.
However, the interpolation-based models provide smaller errors than those
generated by POD, where the better POD model computed by \podavg{} performs
mildly worse than the worst interpolation-based reduced-order models.
Comparing the different interpolation approaches we can observe that mostly,
the interpolation of symmetric transfer functions performs better in the
time simulations and in frequency domain than the interpolation of the 
generalized transfer functions.
For the reduced-order models that provide exact interpolation (\texttt{equi}),
the sampling of higher-order terms in the generalized transfer function needed
to be restricted to match the reduced basis dimension.
This restriction is removed in \texttt{avg} such that more information about
the bilinear and quadratic terms can be obtained in sampling the generalized
transfer functions compared to the symmetric transfer function setting, but
this does only lead to a better reduced-order model when measured with
$\relerr_{\Linf}^{(1)}$.

\Cref{fig:timedelay_time} shows the time simulation of the full-order model
and the best performing reduced-order models from each method over the time
interval $[0, 30]$\,s.
We restricted \Cref{fig:timedelay_time_sim_best} to only the first output
signal for clarity, but the pointwise relative errors in
\Cref{fig:timedelay_time_relerr_best} are computed over both output entries.
For the input signals, we have chosen the mean $\mu = 2$ and the parameter
$\varsigma = 0.25$.
The interpolation-based methods clearly outperform \podavg{} by approximately
four orders of magnitude.
There is no significant difference between the errors of \symintVWavg{}
and \genintVWavg{}.

Similar results can be observed in frequency domain.
The first symmetric subsystem transfer functions are shown in
\Cref{fig:timedelay_g1}, while \Cref{fig:timedelay_g2} illustrates the
pointwise relative errors of the second symmetric subsystem transfer functions.
In both cases, we only show the best performing methods in the frequency
interval $\omega, \omega_{1}, \omega_{2} \in [10^{-3}, 10^{3}]$\,rad/s.
As in the time domain, the interpolation-based methods outperform \podavg{}
by several orders of magnitude in terms of accuracy.
Further plots of the other methods in time and frequency domain can be found in 
the accompanying code package~\cite{supWer23b}.


\subsection{Particle motion in one-dimensional crystal structures}%
\label{subsec:todalattice}

\begin{table}[t]
  \centering
  \caption{Error table of the Toda lattice example:
    The errors are computed as shown in \Cref{subsec:setup}. 
    Only the interpolation methods with one-sided projections provide
    reduced-order models that are stable in the time simulation.
    The models with $\infty$ error are unstable.
    \symintVequi{} outperforms \genintVequi{} by around a factor of $2$.}
  \label{tab:todalattice}
  \vspace{.5\baselineskip}
  
  \begin{tabular}{lrrrr}
    \hline\noalign{\smallskip}
    &
      \multicolumn{1}{c}{$\boldsymbol{\relerr_{L_{2}}}$} &
      \multicolumn{1}{c}{$\boldsymbol{\relerr_{L_{\infty}}}$} & 
      \multicolumn{1}{c}{$\boldsymbol{\relerr_{\Linf}^{(1)}}$} &
      \multicolumn{1}{c}{$\boldsymbol{\relerr_{\Linf}^{(2)}}$} \\
    \noalign{\smallskip}\hline\noalign{\medskip}
    \symintVequi{} &
      $2.4797\texttt{e-}04$ &
      $4.1415\texttt{e-}04$ &
      $4.5518\texttt{e-}03$ &
      $4.6131\texttt{e-}02$ \\
    \symintVavg{} &
      $1.4565\texttt{e-}02$ &
      $1.7264\texttt{e-}02$ &
      $1.4862\texttt{e-}02$ &
      $9.4341\texttt{e-}02$ \\
    \symintVWequi{} &
      $\infty$ &
      $\infty$ &
      $3.7049\texttt{e-}03$ &
      $4.4699\texttt{e-}02$ \\
    \symintVWavg{} &
      $\infty$ &
      $\infty$ &
      $2.3354\texttt{e-}05$ &
      $2.3035\texttt{e-}01$ \\
    \noalign{\medskip}\hline\noalign{\medskip}
    \genintVequi{} &
      $5.0036\texttt{e-}04$ &
      $8.4421\texttt{e-}04$ &
      $1.0227\texttt{e-}02$ &
      $5.6204\texttt{e-}02$ \\
    \genintVavg{} &
      $4.2052\texttt{e-}03$ &
      $8.5704\texttt{e-}03$ &
      $1.8517\texttt{e-}03$ &
      $2.6190\texttt{e-}02$ \\
    \genintVWequi{} &
      $\infty$ &
      $\infty$ &
      $1.3348\texttt{e-}02$ &
      $1.3706\texttt{e+}01$ \\
    \genintVWavg{} &
      $\infty$ &
      $\infty$ &
      $8.2451\texttt{e-}06$ &
      $4.5499\texttt{e-}02$ \\
    \noalign{\medskip}\hline\noalign{\medskip}
    \pod{} &
      $\infty$ &
      $\infty$ &
      $8.1769\texttt{e-}01$ &
      $7.7696\texttt{e-}01$ \\
    \podavg{} &
      $\infty$ &
      $\infty$ &
      $4.6483\texttt{e-}01$ &
      $6.2382\texttt{e-}01$ \\
    \noalign{\medskip}\hline\noalign{\smallskip}
  \end{tabular}
\end{table}

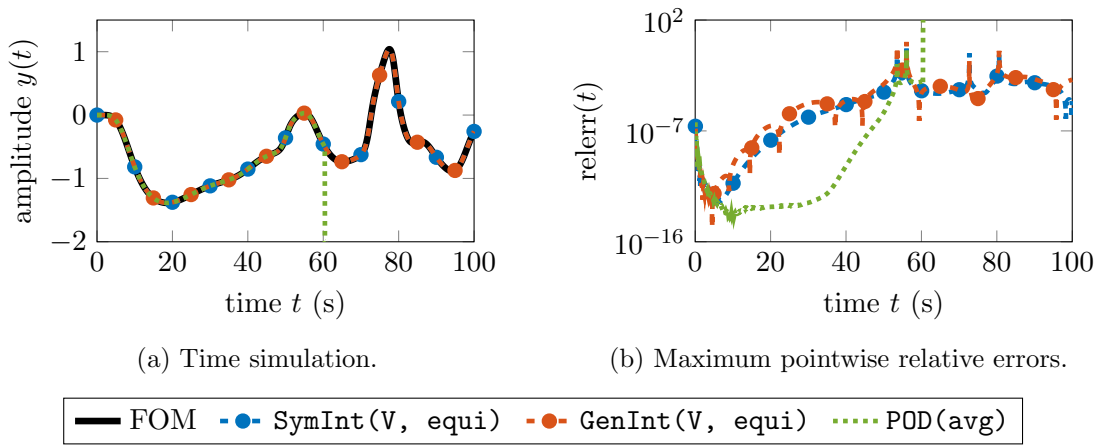
\begin{figure}[t]
  \centering
  \begin{subfigure}[b]{.49\textwidth}
    \centering
  \tikzexternalenable%
  \tikzsetnextfilename{todalattice_time_sim_best}%
  \begin{tikzpicture}
  \pgfplotstableread{graphics/data/todalattice_time_sim_best.dat}\tableSIM
  
  \begin{axis}[%
    width  = .675\textwidth,
    height = .4\textwidth,
    scale only axis,
    xmin = 0,
    xmax = 100,
    ymin = -2.0,
    ymax = 1.5,
    restrict y to domain* = -2.0:1.5,
    xminorticks = false,
    yminorticks = false,
    xlabel = {time $t$ (s)},
    ylabel = {amplitude $y(t)$},
    ylabel style   = {yshift = -.3em},
    scaled x ticks = false,
    x tick label style = {/pgf/number format/fixed},
    clip mode          = individual]
    
    \addplot[fom] table[x index = 0, y index = 1] {\tableSIM};
    \addplot[symint, V, equi, mark repeat = 200]
      table[x index = 0, y index = 2] {\tableSIM};
    \addplot[genint, V, equi, mark repeat = 200, mark phase = 100]
      table[x index = 0, y index = 3] {\tableSIM};
    \addplot[pod, avg] table[x index = 0, y index = 4] {\tableSIM};
  \end{axis}
\end{tikzpicture}%
  \tikzexternaldisable%

    \caption{Time simulation.}
    \label{fig:todalattice_time_sim_best}
  \end{subfigure}%
  \hfill%
  \begin{subfigure}[b]{.49\textwidth}
    \centering
  \tikzexternalenable%
  \tikzsetnextfilename{todalattice_time_relerr_best}%
  \begin{tikzpicture}
  \pgfplotstableread{graphics/data/todalattice_time_relerr_best.dat}\tableERR
  
  \begin{semilogyaxis}[%
    width  = .675\textwidth,
    height = .4\textwidth,
    scale only axis,
    xmin = 0,
    xmax = 100,
    ymin = 1e-16,
    ymax = 1e+2,
    xminorticks = false,
    yminorticks = false,
    xlabel = {time $t$ (s)},
    ylabel = {$\relerr(t)$},
    ylabel style   = {yshift = -.3em},
    scaled x ticks = false,
    x tick label style = {/pgf/number format/fixed},
    clip mode          = individual]
    
    \addplot[symint, V, equi, mark repeat = 200]
      table[x index = 0, y index = 1] {\tableERR};
    \addplot[genint, V, equi, mark repeat = 200, mark phase = 100]
      table[x index = 0, y index = 2] {\tableERR};
    \addplot[pod, avg] table[x index = 0, y index = 3] {\tableERR};
  \end{semilogyaxis}
\end{tikzpicture}%
  \tikzexternaldisable%

    \caption{Maximum pointwise relative errors.}
    \label{fig:todalattice_time_relerr_best}
  \end{subfigure}
  
  \vspace{.5\baselineskip}
  \tikzexternalenable%
  \tikzsetnextfilename{todalattice_time_legend}%
  \begin{tikzpicture}
  \begin{axis}[%
    hide axis,
    width  = 1mm,
    height = 1mm,
    scale only axis,
    xmin = 0,
    xmax = 1,
    ymin = 0,
    ymax = 1,
    legend columns = 4, 
    legend style   = {
      at     = {(0,0)},
      anchor = center,
      /tikz/every even column/.append style = {column sep = 0.2cm}},
    legend cell align  = {left},
    clip mode          = individual]
    
    \addlegendimage{fom}
    \addlegendentry{FOM}
    
    \addlegendimage{symint, V, equi}
    \addlegendentry{\symintVequi}
    
    \addlegendimage{genint, V, equi}
    \addlegendentry{\genintVequi}
    
    \addlegendimage{pod, avg}
    \addlegendentry{\podavg}
  \end{axis}
\end{tikzpicture}%
  \tikzexternaldisable%

  \caption{Time simulation of the Toda lattice example:
    The best reduced-order models from each generating approach are shown.
    Only the interpolation-based reduced-order models recover the system
    behavior over the full time interval, while \podavg{} becomes unstable
    after about $60$\,s.}
  \label{fig:todalattice_time}
\end{figure}

\begin{figure}[t]
  \centering
  \begin{subfigure}[b]{.49\textwidth}
    \centering
  \tikzexternalenable%
  \tikzsetnextfilename{todalattice_g1_tf_best}%
  \begin{tikzpicture}
  \pgfplotstableread{graphics/data/todalattice_g1_tf_best.dat}\tableTF
  
  \begin{loglogaxis}[%
    width  = .675\textwidth,
    height = .4\textwidth,
    scale only axis,
    xmin = 1e-2,
    xmax = 1e+2,
    ymin = 1e-16,
    ymax = 1e+2,
    xminorticks = false,
    yminorticks = false,
    xlabel = {frequency $\omega$ (rad/s)},
    ylabel = {magnitude},
    ylabel style   = {yshift = -.3em},
    scaled x ticks = false,
    x tick label style = {/pgf/number format/fixed},
    clip mode          = individual]
    
    \addplot[fom] table[x index = 0, y index = 1] {\tableTF};
    \addplot[symint, V, equi, mark repeat = 50]
      table[x index = 0, y index = 2] {\tableTF};
    \addplot[genint, V, equi, mark repeat = 50, mark phase = 20]
      table[x index = 0, y index = 3] {\tableTF};
    \addplot[pod, avg] table[x index = 0, y index = 4] {\tableTF};
  \end{loglogaxis}
\end{tikzpicture}%
  \tikzexternaldisable%

    \caption{Linear transfer functions.}
    \label{fig:todalattice_g1_tf_best}
  \end{subfigure}%
  \hfill%
  \begin{subfigure}[b]{.49\textwidth}
    \centering
  \tikzexternalenable%
  \tikzsetnextfilename{todalattice_g1_relerr_best}%
  \begin{tikzpicture}
  \pgfplotstableread{graphics/data/todalattice_g1_relerr_best.dat}\tableERR
  
  \begin{loglogaxis}[%
    width  = .675\textwidth,
    height = .4\textwidth,
    scale only axis,
    xmin = 1e-2,
    xmax = 1e+2,
    ymin = 1e-14,
    ymax = 1e+0,
    xminorticks = false,
    yminorticks = false,
    xlabel = {frequency $\omega$ (rad/s)},
    ylabel = {$\relerr(\omega)$},
    ylabel style   = {yshift = -.3em},
    scaled x ticks = false,
    x tick label style = {/pgf/number format/fixed},
    clip mode          = individual]
    
    \addplot[symint, V, equi, mark repeat = 50]
      table[x index = 0, y index = 1] {\tableERR};
    \addplot[genint, V, equi, mark repeat = 50, mark phase = 20]
      table[x index = 0, y index = 2] {\tableERR};
    \addplot[pod, avg] table[x index = 0, y index = 3] {\tableERR};
  \end{loglogaxis}
\end{tikzpicture}%
  \tikzexternaldisable%

    \caption{Pointwise relative errors.}
    \label{fig:todalattice_g1_relerr_best}
  \end{subfigure}
  
  \vspace{.5\baselineskip}
  \tikzexternalenable%
  \tikzsetnextfilename{todalattice_g1_legend}%
  \begin{tikzpicture}
  \begin{axis}[%
    hide axis,
    width  = 1mm,
    height = 1mm,
    scale only axis,
    xmin = 0,
    xmax = 1,
    ymin = 0,
    ymax = 1,
    legend columns = 4, 
    legend style   = {
      at     = {(0,0)},
      anchor = center,
      /tikz/every even column/.append style = {column sep = 0.2cm}},
    legend cell align  = {left},
    clip mode          = individual]
    
    \addlegendimage{fom}
    \addlegendentry{FOM}
    
    \addlegendimage{symint, V, equi}
    \addlegendentry{\symintVequi}
    
    \addlegendimage{genint, V, equi}
    \addlegendentry{\genintVequi}
    
    \addlegendimage{pod, avg}
    \addlegendentry{\podavg}
  \end{axis}
\end{tikzpicture}%
  \tikzexternaldisable%

  \caption{First symmetric transfer function  of the Toda lattice example:
    Only the best performing methods from the time simulation are shown.
    All reduced-order models recover the transfer function behavior of the
    original system.
    For low frequencies, \podavg{} performs two order of magnitude worse than
    the interpolating methods.}
  \label{fig:todalattice_g1}
\end{figure}
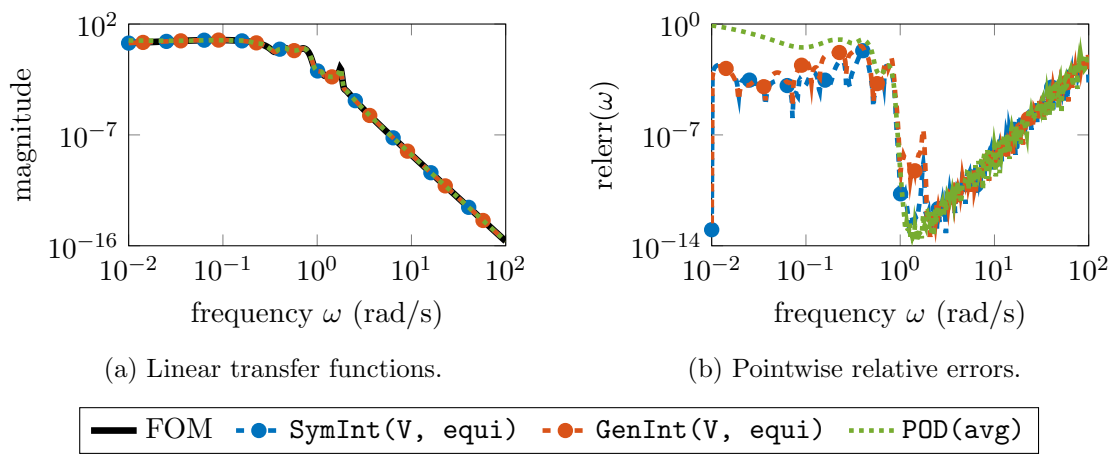

\begin{figure}[t]
  \centering
  \begin{subfigure}[b]{.49\textwidth}
    \centering
  \tikzexternalenable%
  \tikzsetnextfilename{todalattice_g2_symvequi}%
  \begin{tikzpicture}
  \begin{loglogaxis}[
    view   = {0}{90},
    width  = .675\textwidth,
    height = .4\textwidth,
    scale only axis,
    axis on top,
    xmin   = 1e-2,
    xmax   = 1e+2,
    ymin   = 1e-2,
    ymax   = 1e+2,
    xtick  = {1e-2, 1e-1, 1e+0, 1e+1, 1e+2},
    ytick  = {1e-2, 1e-1, 1e+0, 1e+1, 1e+2},
    xminorticks = false,
    yminorticks = false,
    xlabel = {frequency $\omega_{1}$ (rad/s)},
    ylabel = {frequency $\omega_{2}$ (rad/s)},
    ylabel style = {yshift = -.3em},
    scaled x ticks = false,
    x tick label style = {/pgf/number format/fixed}]
        
      \addplot graphics[xmin = 1e-2, xmax = 1e+2, ymin = 1e-2, ymax = 1e+2]
        {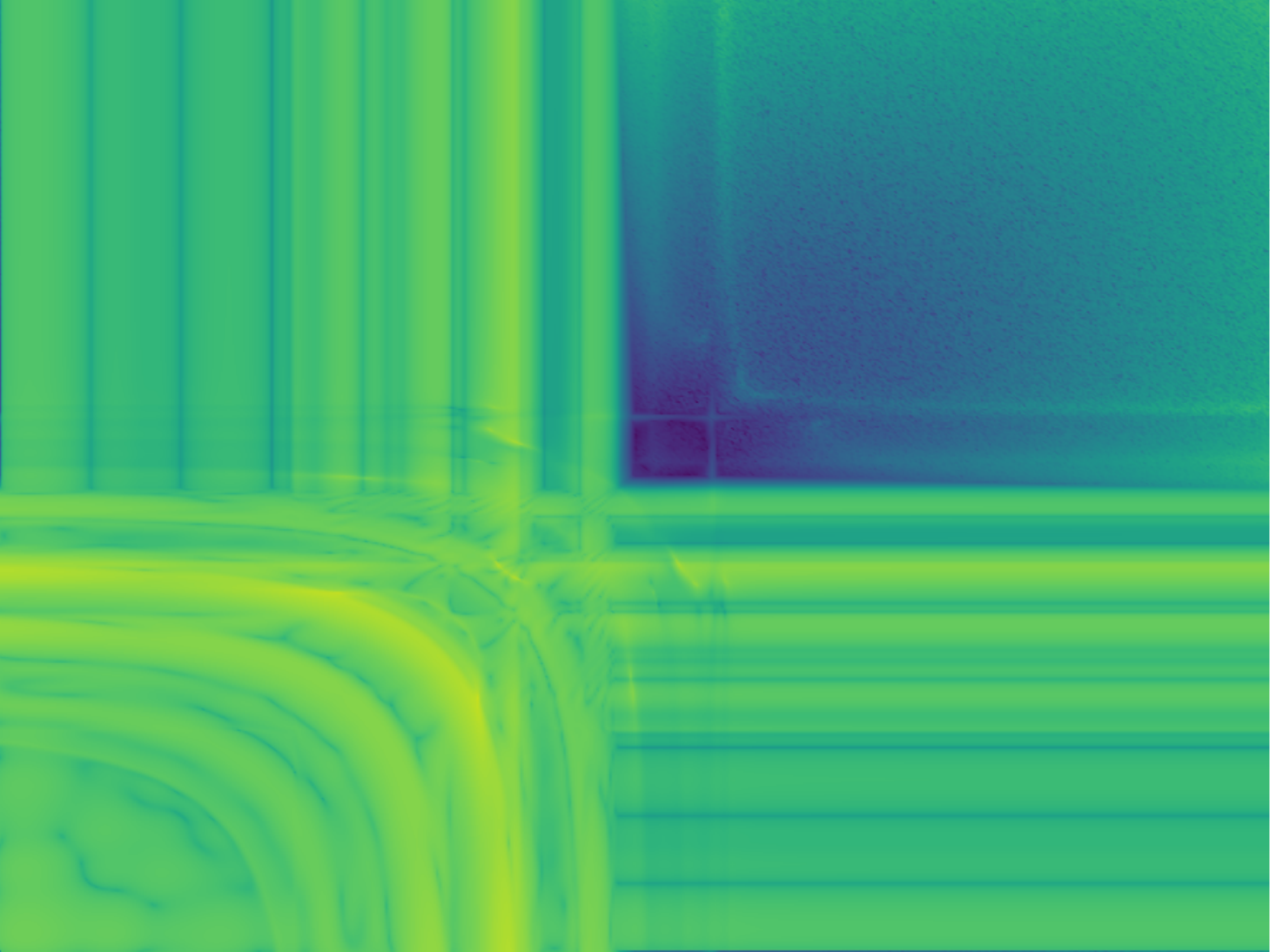};
            
  \end{loglogaxis}
\end{tikzpicture}%
  \tikzexternaldisable%

    \caption{\symintVequi{}.}
    \label{fig:todalattice_g2_symvequi}
  \end{subfigure}%
  \hfill%
  \begin{subfigure}[b]{.49\textwidth}
    \centering
  \tikzexternalenable%
  \tikzsetnextfilename{todalattice_g2_genvequi}%
  \begin{tikzpicture}
  \begin{loglogaxis}[
    view   = {0}{90},
    width  = .675\textwidth,
    height = .4\textwidth,
    scale only axis,
    axis on top,
    xmin   = 1e-2,
    xmax   = 1e+2,
    ymin   = 1e-2,
    ymax   = 1e+2,
    xtick  = {1e-2, 1e-1, 1e+0, 1e+1, 1e+2},
    ytick  = {1e-2, 1e-1, 1e+0, 1e+1, 1e+2},
    xminorticks = false,
    yminorticks = false,
    xlabel = {frequency $\omega_{1}$ (rad/s)},
    ylabel = {frequency $\omega_{2}$ (rad/s)},
    ylabel style = {yshift = -.3em},
    scaled x ticks = false,
    x tick label style = {/pgf/number format/fixed}]
        
      \addplot graphics[xmin = 1e-2, xmax = 1e+2, ymin = 1e-2, ymax = 1e+2]
        {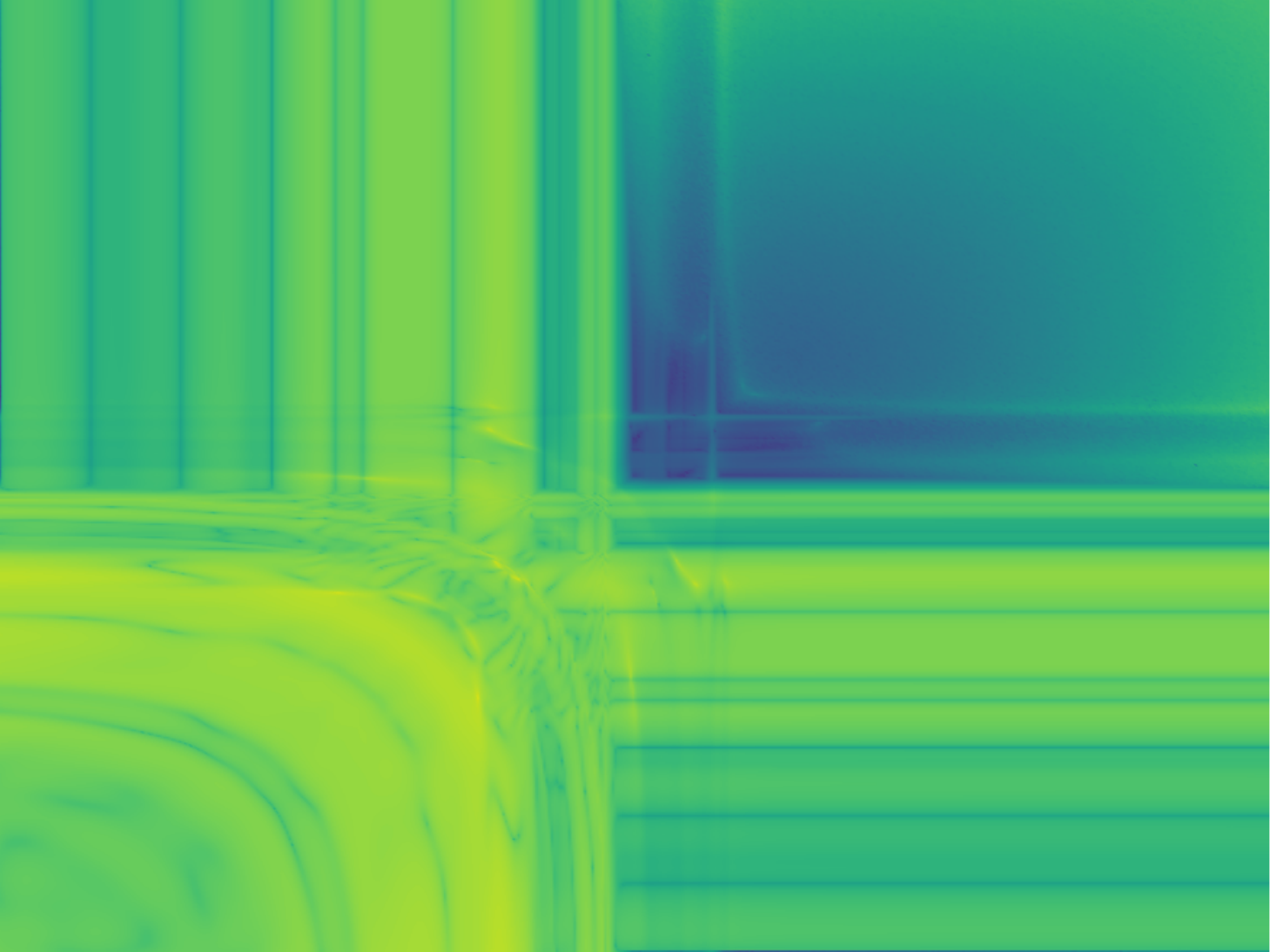};
            
  \end{loglogaxis}
\end{tikzpicture}%
  \tikzexternaldisable%

    \caption{\genintVequi{}.}
    \label{fig:todalattice_g2_genvequi}
  \end{subfigure}
  
  \vspace{.5\baselineskip}
  \begin{subfigure}{.49\textwidth}
    \centering
  \tikzexternalenable%
  \tikzsetnextfilename{todalattice_g2_podavg}%
  \begin{tikzpicture}
  \begin{loglogaxis}[
    view   = {0}{90},
    width  = .675\textwidth,
    height = .4\textwidth,
    scale only axis,
    axis on top,
    xmin   = 1e-2,
    xmax   = 1e+2,
    ymin   = 1e-2,
    ymax   = 1e+2,
    xtick  = {1e-2, 1e-1, 1e+0, 1e+1, 1e+2},
    ytick  = {1e-2, 1e-1, 1e+0, 1e+1, 1e+2},
    xminorticks = false,
    yminorticks = false,
    xlabel = {frequency $\omega_{1}$ (rad/s)},
    ylabel = {frequency $\omega_{2}$ (rad/s)},
    ylabel style = {yshift = -.3em},
    scaled x ticks = false,
    x tick label style = {/pgf/number format/fixed}]
        
      \addplot graphics[xmin = 1e-2, xmax = 1e+2, ymin = 1e-2, ymax = 1e+2]
        {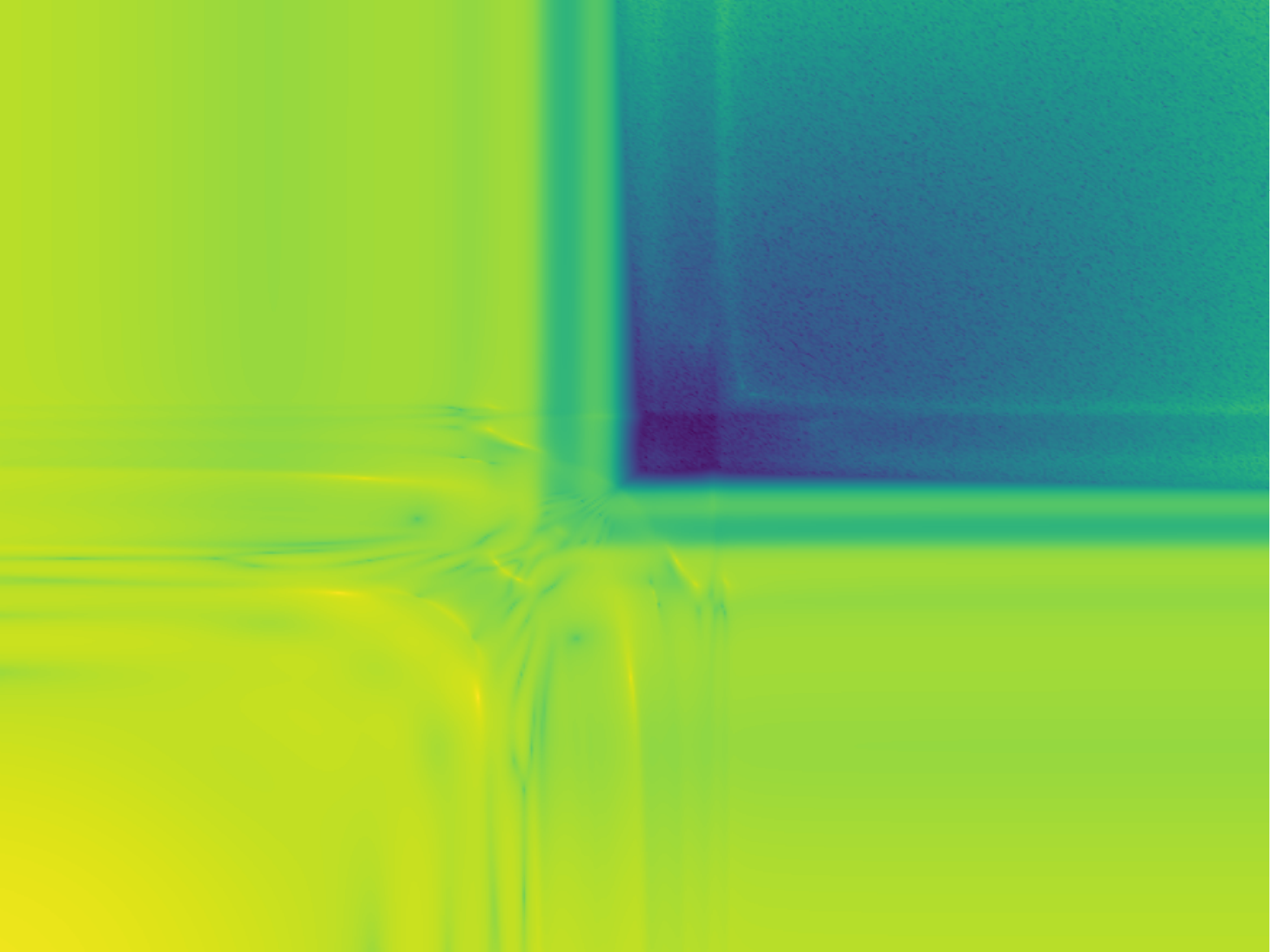};
            
  \end{loglogaxis}
\end{tikzpicture}%
  \tikzexternaldisable%

    \caption{\podavg{}.}
    \label{fig:todalattice_g2_podavg}
  \end{subfigure}%
  \hfill%
  \begin{subfigure}{.49\textwidth}
    \raggedleft
  \tikzexternalenable%
  \tikzsetnextfilename{todalattice_g2_legend}%
  \begin{tikzpicture}
  \node[draw = none, minimum width = 0cm, inner sep = 0cm](start){};
  \node(leg) at (start.north east) [anchor = north west]{\tikz
  \begin{axis}[%
    hide axis,
    scale only axis,
    width  = .8\textwidth,
    height = .1cm,
    point meta min = -13.7278,
    point meta max = 0.8981,
    colorbar,
    colorbar horizontal,
    colorbar style = {
      xticklabel = $10^{\pgfmathparse{\tick}
        \pgfmathprintnumber\pgfmathresult}$,
      at     = {(.5, 0)},
      anchor = north},
    scaled x ticks = false,
    x tick label style = {/pgf/number format/fixed}]
  \end{axis};};
  \node[draw = none, minimum width = 0cm, inner sep = 0cm](end)
    at (leg.north east) [anchor = north west]{};
\end{tikzpicture}%
  \tikzexternaldisable%

    \vspace{2\baselineskip}
  \end{subfigure}
  
  \caption{Second symmetric transfer function relative approximation errors
    $\relerr(\omega_{1}, \omega_{2})$ of the Toda lattice example:
    Only the best performing methods from the time simulation are shown.
    For large frequencies, \symintVequi{} and \podavg{} perform equally
    well and better than \genintVequi{}, while for low frequencies, \podavg{}
    is an order of magnitude worse than the interpolating methods.}
  \label{fig:todalattice_g2}
\end{figure}

As second example, we consider the motion of particles in a one-dimensional
crystal structure described by the Toda lattice model from the introduction; see
\Cref{fig:todalattice}.
The original system with exponential nonlinearities~\cref{eqn:todalattice} can
be rewritten into quadratic-bilinear form by introducing the auxiliary variables
\begin{equation} \label{eqn:auxvar}
  z_{j}(t) = \begin{cases}
    e^{k_{j}(q_{j}(t) - q_{j + 1}(t))} - 1 & \text{for}~j < \ell,\\
    e^{k_{\ell}q_{\ell}(t)} - 1 & \text{for}~j = \ell.
  \end{cases}
\end{equation}
By differentiating~\cref{eqn:auxvar} twice, the Toda lattice model can be written
as a system of quadratic-bilinear ordinary differential equations of the form
\begin{equation} \label{eqn:todalatticesq}
  \begin{aligned}
    0 & = M \ddot{q}(t) + D \dot{q}(t) + K q(t) +
      \Hvv\big(\dot{q}(t) \otimes \dot{q}(t)\big) +
      \Hpv\big(q(t) \otimes \dot{q}(t)\big)\\
    & \quad{}+{} 
      \Hpp\big(q(t) \otimes q(t)\big) -
      \Np q(t) u(t) - \Bu u(t),\\
    y(t) & = \Cv \dot{q}(t),
  \end{aligned}
\end{equation}
with the dimensions as in~\cref{eqn:sqsys} and $m = 1$ input and $p=1$ output.
The exact parameterization of the matrices in~\cref{eqn:todalatticesq}
can be found in~\cite[Sec.~6.5]{Wer21}.
For our experiments, we use the same setup as in~\cite[Sec.~6.5]{Wer21}
with $\ell = 2\,000$ particles such that~\cref{eqn:todalatticesq} has the
order~$n = 4\,000$.

It has been observed in~\cite{Wer21} that the internal block structures of the
matrices in~\cref{eqn:todalatticesq} resulting from the original and auxiliary
variables should be preserved for stability of the reduced-order models.
Therefore, we follow the suggestion in~\cite[Sec.~6.5]{Wer21} and use the
\emph{split congruence transformation}
approach~\cite{Fre05, RitWBetal20, VanNDetal17}.
That is, given a basis matrix $V = \begin{bmatrix} V_{1}^{\herm} &
V_{2}^{\herm} \end{bmatrix}^{\herm} \in \C^{2 \ell \times r}$, we construct the
extended basis matrix
\begin{equation*}
  \tV = \begin{bmatrix} V_{1} & 0 \\ 0 & V_{2} \end{bmatrix} \in 
    \C^{2\ell \times 2r},
\end{equation*}
and similarly for a left basis matrix $W$.
The extended basis matrices are then used for model reduction by projection.
We apply this approach in our experiments to modify all projection basis
matrices computed as described in \Cref{subsec:setup}.
By construction, it holds that
\begin{equation*}
  \mspan(V) \subseteq \mspan(\tV).
\end{equation*}
Therefore, if $V$ was constructed to satisfy any subspace conditions in
\Cref{sec:interp} for interpolation, the basis matrix $\tV$ also satisfies these
conditions such that interpolation properties are preserved.

For the comparison in our experiments, we have chosen the reduced order of all
computed models to be $2 r = 120$.
The results are shown in \Cref{tab:todalattice}.
The best performing model in terms of time simulation error is \symintVequi{}
followed by its counterpart for the generalized transfer functions
\genintVequi{}.
None of the models resulting from a two-sided projection has a stable
time-domain simulation, which appears to be a consequence of loosing additional
mechanical properties by $V \neq W$.
Also none of the POD generated models performs stable in the time-domain
simulation.
Here, the large frequency domain errors indicate that the approximation quality
is not sufficient to approximate the system behavior well enough.
In frequency domain, we observe similarly to the previous numerical example
that \symintVWavg{} and \genintVWavg{} perform best in terms of
approximating the first symmetric subsystem transfer function.

The time-domain simulations of the full- and the best performing reduced-order
models from each method are shown in \Cref{fig:todalattice_time} in the time
interval $[0, 100]$\,s.
For the input signal, we have chosen the mean $\mu = 0$ and the smoothing
parameter $\varsigma = 2$.
\podavg{} performs visibly stable only until around $60$\,s, while
the other two models follow the system behavior over the complete time interval.
The pointwise relative errors in \Cref{fig:todalattice_time_relerr_best} reveal
\podavg{} to be more accurate than \symintVequi{} and \genintVequi{} for the
first half of the time interval before it assumes the same error level as the
other two methods and finally becomes unstable.
\symintVequi{} and \genintVequi{} have overall a similar error behavior, with
the errors of \genintVequi{} being mildly larger.

On the other hand, in frequency domain, we can observe a similar behavior
compared to the previous numerical example.
The results for the same reduced-order models that performed best in
time domain can be seen in \Cref{fig:todalattice_g1,fig:todalattice_g2},
with the frequency interval of interest
$\omega, \omega_{1}, \omega_{2} \in [10^{-3}, 10^{3}]$\,rad/s.
The POD generated models perform worst, with the exception of \genintVWequi{}.
The models based on oversampling generalized transfer functions perform
better than those based on oversampling the symmetric transfer functions due to
the additional information obtained from the nonlinear terms.
However, in this example we can observe that the oversampling procedure may
produce larger errors than exact interpolation.
In particular, models computed via \texttt{avg} provide worse
approximation errors for larger frequencies, where the transfer functions are
converging to zero, while the models with \texttt{equi} preserve the system
behavior due to the exact interpolation.
Further plots of the other methods in time and frequency domain can be found in 
the accompanying code package~\cite{supWer23b}.


\section{Conclusions}%
\label{sec:conclusions}

We have extended the structure-preserving interpolation framework to
quadratic-bilinear systems.
Based on two motivating structured examples, we have introduced 
the structured variants of the symmetric subsystem and generalized transfer
functions of quadratic-bilinear systems.
For both transfer function types, we provided subspace conditions enabling the
computation of interpolating structured reduced-order models by projection.
The theoretical findings are then used to compute structured reduced-order
models in two numerical examples.
The theory presented here can be applied to a much broader class of structures
than those used here for illustrations.

The numerical results suggest that the interpolation of symmetric transfer
functions provides more accurate reduced-order models than the generalized
transfer functions for the same reduced order.
This is most certainly a consequence of the restriction to only products of
system terms in the generalized transfer functions.
However, we have seen that when the basis matrices are first constructed by
oversampling and then compressed, the generalized transfer function
interpolation framework provides more accurate reduced-order models for the
same computational costs as for the symmetric transfer functions due to
additional information obtained from the nonlinear terms.
The authors of~\cite{GoyPB23} extend on this oversampling idea and the
definition of structured generalized transfer functions for quadratic-bilinear
systems from this paper to propose structure-preserving model reduction for
systems with polynomial nonlinearities based on the interpolation of
generalized transfer functions with at most one nonlinear component.
While this gives a first efficient approach to simulation-free model
reduction for polynomial systems, there are many open questions left.
One related to our observations in this work is the question whether exact
interpolation of transfer functions based on Volterra kernels may perform 
better than an oversampling procedure.
This needs a thorough investigation of interpolation conditions
for polynomial systems, which we will address in some future work.

Another transfer function type for quadratic-bilinear systems are regular
subsystem transfer functions.
These have been omitted in this paper since for the choice of identical
interpolation points, the projection spaces of regular and symmetric subsystem
transfer functions coincide.
The formulas for structured variants of regular transfer functions and
results on interpolation conditions will be presented in a separate 
work.

For simplicity of exposition, we have restricted the numerical experiments to
only logarithmically equidistant interpolation points on the imaginary axis.
While such a procedure is often sufficient in practice, the question of good 
or even optimal interpolation points remains open and crucial for the success
of such model reduction methods.
This is still an unresolved issue even in the case of structured linear systems
and needs further investigation in the future.


\section*{Acknowledgments}%
\addcontentsline{toc}{section}{Acknowledgments}

Parts of this work were carried out while Werner was at the Max Planck
Institute for Dynamics of Complex Technical Systems in Magdeburg, Germany.

Benner and Werner have been supported been supported by the German Research
Foundation (DFG) Research Training Group 2297 ``Mathematical Complexity
Reduction (MathCoRe)''.
The work of Gugercin is based upon work supported by the National Science
Foundation under Grant No. DMS--1819110.


\addcontentsline{toc}{section}{References}
\bibliographystyle{plainurl}
\bibliography{bibtex/myref}

\end{document}